\documentclass[10pt,reqno]{amsart}
\usepackage{amssymb,mathrsfs,graphicx}
\usepackage{ifthen}
\usepackage{hyperref}
\usepackage[margin=1in]{geometry}
\usepackage{caption}
\usepackage{sidecap}
\usepackage{rotating,dsfont}
\usepackage{colortbl}
\definecolor{black}{rgb}{0.0, 0.0, 0.0}
\definecolor{red}{rgb}{1.0, 0.5, 0.5}
\provideboolean{shownotes} 
\setboolean{shownotes}{true} 
\newcommand{\margnote}[1]{
\ifthenelse{\boolean{shownotes}}%
{\marginpar{\raggedright\tiny\texttt{#1}}}%
{}%
}
\newcommand{\hole}[1]{
\ifthenelse{\boolean{shownotes}}%
{\begin{center} \fbox{ \rule {.25cm}{0cm} \rule[-.1cm]{0cm}{.4cm}
\parbox{.85\textwidth}{\begin{center} \texttt{#1}\end{center}} \rule
{.25cm}{0cm}}\end{center}} {} }

\title[Modulated energy estimates for singular kernels]{Modulated energy estimates for singular kernels and their applications to asymptotic analyses for kinetic equations}

\author[Choi]{Young-Pil Choi}
\address[Young-Pil Choi]{\newline Department of Mathematics\newline
Yonsei University, 50 Yonsei-Ro, Seodaemun-Gu, Seoul 03722, Republic of Korea}
\email{ypchoi@yonsei.ac.kr}

\author[Jung]{Jinwook Jung}
\address[Jinwook Jung]{\newline Research Institute of Basic Sciences \newline Seoul National University, Seoul  08826, Republic of Korea}
\email{warp100@snu.ac.kr}

\numberwithin{equation}{section}

\newtheorem{theorem}{Theorem}[section]
\newtheorem{lemma}{Lemma}[section]

\newtheorem{proposition}{Proposition}[section]
\newtheorem{remark}{Remark}[section]

\newcommand{\R}{\mathbb R}

\newcommand{\om}{\Omega}
\newcommand{\bbn}{\mathbb N}
\newcommand{\N}{\mathbb N}

\newcommand{\mc}{\mathcal C}
\newcommand{\K}{\mathcal K}

\newcommand{\bq}{\begin{equation}}
\newcommand{\eq}{\end{equation}}
\newcommand{\e}{\varepsilon}
\newcommand{\lt}{\left}
\newcommand{\rt}{\right}

\newcommand{\pa}{\partial}

\newcommand{\mi}{\mathcal{I}}
\newcommand{\mh}{\mathcal{H}}
\newcommand{\me}{\mathcal{E}}
\newcommand{\mf}{\mathcal{F}}
\newcommand{\md}{\mathcal{D}}

\newcommand{\intr}{\int_{\R^d}}

\newcommand{\intrr}{\iint_{\R^d \times \R^d}}

\newcommand{\sfI}{\mathsf{I}}
\newcommand{\sfJ}{\mathsf{J}}
\newcommand{\sfK}{\mathsf{K}}

\def\d{\mathrm{d}}

\begin{document}
\allowdisplaybreaks

\date{\today}


\begin{abstract}In this paper, we provide modulated interaction energy estimates for the kernel $K(x) = |x|^{-\alpha}$ with $\alpha \in (0,d)$, and its applications to quantified asymptotic analyses for kinetic equations. The proof relies on a dimension extension argument for an elliptic operator and its commutator estimates. For the applications, we first discuss the quantified small inertia limit of kinetic equation with singular nonlocal interactions. The aggregation equations with singular interaction kernels are rigorously derived. We also study the rigorous quantified hydrodynamic limit of the kinetic equation to derive the isothermal Euler or pressureless Euler system with the nonlocal singular interactions forces. 
\end{abstract}

\maketitle \centerline{\date}


%
%
%
%
\section{Introduction}\label{sec:intro}

This paper is devoted to estimates of modulated interaction energy with singular kernels and its applications to asymptotic analyses for kinetic equations. To be precise, we are interested in the  singular kernel $K: \R^d \to \R_+$, which has the form of  
\bq\label{kernel}
K(x) = \frac{1}{|x|^\alpha} \quad \mbox{with } \alpha \in (0,d).
\eq
In this setting, $K$ is called Coulomb potential when $\alpha = d-2$ with $d\geq 3$ and the case $0 \vee (d-2) \leq \alpha < d$ is known as Riesz potential.   The singular kernel $K$ is usually taken into account the nonlocal interaction forces in charge of the repulsion or attraction between interacting particles. The interaction energy through the singular kernel $K$ is typically given by 
\[
\intr \mi_K(\rho)\,dx := \frac12 \intr \rho (K \star \rho)\,dx
\]
for some density function $\rho$, and thus its modulated interaction energy is naturally defined as
\[
\intr \mi_K(\rho|\bar\rho)\,dx := \intr \lt(\mi_K(\rho) - \mi_K(\bar\rho) - (D_\rho \mi_K )(\bar\rho)(\rho - \bar\rho)\rt) dx = \frac12 \intr (\rho - \bar\rho) K \star (\rho - \bar\rho)\,dx.
\]

The main purpose of this work is to estimate the above modulated interaction energy associated with the continuity equation.
\begin{theorem}\label{thm_main}
Let $T>0$ and $K$ be given by \eqref{kernel}. Suppose that the pairs $(\bar\rho, \bar u)$ and $(\rho,u)$ satisfy the followings:
\begin{enumerate}
\item[(i)]
$(\bar\rho, \bar u)$ and $(\rho,u)$ satisfy the continuity equations in the sense of distribution:
\[
\pa_t \bar\rho + \nabla \cdot (\bar\rho \bar u) =0 \quad \mbox{and} \quad \pa_t \rho + \nabla \cdot (\rho u) =0,
\]
\item[(ii)]
$(\bar\rho, \bar u)$ and $(\rho,u)$ satisfy the energy inequality:
\[
\sup_{0\le t \le T}\lt( \intr \bar\rho |\bar u|^2\,dx + \intr \bar\rho K\star\bar\rho\,dx\rt) <\infty, \quad \sup_{0\le t \le T}\lt( \intr \rho | u|^2\,dx + \intr \rho K\star\rho\,dx\rt) <\infty,
\]
\item[(iii)]
$\bar\rho, \rho\in \mc(0,T; L^1(\R^d))$, $\nabla u\in L^\infty(\R^d \times (0,T))$ and if $\alpha<d-2$, 
\bq\label{cond_u}
\left\{\begin{array}{lcl}
\nabla^{[(d-\alpha)/2]+1} u\in L^\infty(0,T;L^{\frac{d}{[(d-\alpha)/2]}}(\R^d)) & \mbox{if} & \alpha\in(0,d-2)\setminus(d-2\N),\\
\nabla^{\frac{d-\alpha}{2}} u \in L^\infty(0,T;L^{\frac{2d}{d-\alpha-2}}(\R^d)) & \mbox{if} & \alpha\equiv d \ \emph{ mod } 2, 
\end{array}\rt.
\eq
where $d-2\N := \{d-2n \ : \ n \in \N\}$ and $[ \,\cdot \, ]$ denotes the floor function. 
\end{enumerate}
Then we have
\bq\label{res_thm}
\frac12\frac{d}{dt}\intr (\rho -\bar\rho) K \star (\rho - \bar\rho)\,dx \leq \intr \bar\rho(u - \bar u) \cdot \nabla K \star (\rho - \bar\rho)\,dx + C\intr (\rho - \bar\rho) K \star (\rho - \bar\rho)\,dx
\eq
for $t \in [0,T)$ and some $C>0$ which depends only on $\alpha$, $d$ and $\|\nabla u\|_{L^\infty(\R^d \times (0,T))}$, and  if $d < \alpha-2$, additionally  
\[
\lt\{\begin{array}{lcl}
\|\nabla^{[(d-\alpha)/2]+1} u\|_{L^\infty(0,T;L^{\frac{d}{[(d-\alpha)/2]-1}})} & \mbox{if} & \alpha\in(0,d-2)\setminus (d-2\N),\\
\|\nabla^{(d-\alpha)/2}  u\|_{L^\infty(0,T;L^{\frac{2d}{d-\alpha-2}})} & \mbox{if} & \alpha\equiv d \ \emph{mod } 2.
\end{array}\rt.\]
\end{theorem}

\begin{remark}\label{rmk_ip} Note that 
\begin{align*}
\frac12\frac{d}{dt}\intr (\rho -\bar\rho) K \star (\rho - \bar\rho)\,dx &= \intr \pa_t (\rho -\bar\rho) K \star (\rho - \bar\rho)\,dx\cr
&=\intr (\rho u - \bar\rho \bar u) \cdot \nabla K \star (\rho - \bar\rho)\,dx\cr
&=  \intr \bar\rho(u - \bar u) \cdot \nabla K \star (\rho - \bar\rho)\,dx\cr
&\quad + \intr (\rho-\bar\rho) u \cdot \nabla K \star(\rho-\bar\rho)\,dx.
\end{align*}
This implies that to prove Theorem \ref{thm_main} it suffices to show 
\bq\label{ess}
\intr (\rho-\bar\rho) u \cdot \nabla K \star(\rho-\bar\rho)\,dx \leq C\intr (\rho - \bar\rho) K \star (\rho - \bar\rho)\,dx
\eq
for $C>0$ which depends only on $d$, $\|\nabla u\|_{L^\infty(\R^d \times (0,T))}$ and  if $d < \alpha-2$, additionally  
\[
\lt\{\begin{array}{lcl}
\|\nabla^{[(d-\alpha)/2]+1} u\|_{L^\infty(0,T;L^{\frac{d}{[(d-\alpha)/2]-1}})} & \mbox{if} & \alpha\in(0,d-2)\setminus (d-2\N),\\
\|\nabla^{(d-\alpha)/2}  u\|_{L^\infty(0,T;L^{\frac{2d}{d-\alpha-2}})} & \mbox{if} & \alpha\equiv d \ \emph{mod } 2.
\end{array}\rt.\]
We would like to point out that the right-hand side of the above is nonnegative since $K$ is positive definite. 
\end{remark}

\begin{remark}
Actually, the condition \eqref{cond_u} can be slightly relaxed to
\[
\lt\{\begin{array}{lcl}
&\nabla^{\lt[\frac{d-\alpha}{2}\rt]}u \in L^\infty(0,T;L^{\frac{d}{[(d-\alpha)/2]-1}}(\R^d)), \quad \Delta^{\lt[\frac{d-\alpha}{4}\rt]} u \in L^\infty(0,T;L^{\frac{d}{[(d-\alpha)/2]}}(\R^d))  \cr
&\mbox{if} \quad  0\vee(4m-2)<d-\alpha<4m, \quad m \in \N,\\
&\nabla^{\lt[\frac{d-\alpha}{2}\rt]}u \in L^\infty(0,T;L^{\frac{d}{[(d-\alpha)/2]-1}}(\R^d)), \quad \nabla\Delta^{\lt[\frac{d-\alpha}{4}\rt]} u \in L^\infty(0,T;L^{\frac{d}{[(d-\alpha)/2]}}(\R^d)) \cr
& \mbox{if} \quad 0\vee 4m<d-\alpha<4m+2, \quad m \in \N,\\
&\nabla^{\frac{d-\alpha}{2}} u \in L^\infty(0,T;L^{\frac{2d}{d-\alpha-2}}(\R^d)) \quad \mbox{if} \quad \alpha\equiv d \ \emph{ mod } 2.
\end{array}\rt.
\]
We simply used the Gagliardo--Nirenberg interpolation inequality \eqref{gn_ineq} to reduce the first two conditions above into $\nabla^{[(d-\alpha)/2]+1} u\in L^\infty(0,T;L^{\frac{d}{[(d-\alpha)/2]}}(\R^d))$ for simplicity. Moreover, we also note that the conditions on $u$ in Theorems \ref{thm_ktoc} and \ref{thm_hydro} can be substituted as above. For detail, see the proof of Theorem \ref{thm_main} in Section \ref{sec_subCoul}.
\end{remark}

The estimates of modulated energies and its variants have been applied to the study for the quantified mean-field limits \cite{BJWpre, Due16, JW18, NRSpre, Ser17, Ser20}, relaxation limit \cite{C21, CJpre1, LT13,LT17}, overdamped limit \cite{CTpre}, hydrodynamic limits \cite{CCJ21, HI21}. In a very recent paper \cite{NRSpre}, the modulated interaction energy estimate with the kernel $K$ given by \eqref{kernel} is discussed to study the mean-field limit from many interacting particle systems towards the continuity equation. 


\subsection{Ideas of the proof}
The main idea of our strategy consists of two steps. We first interpret the singular kernel $K$ as a fundamental solution to certain partial differential equations in higher dimensions by dealing with dummy variables if necessary. This enables us to rewrite the left hand side of \eqref{ess}, and it eventually leads to commutator estimates. Let us explain it more specifically by dividing into two cases: $\alpha \in (0 \vee (d-2), d)$ and $\alpha \in (0, d-2)$.

The modulated interaction energy estimate for Riesz interaction potential case, i.e., \eqref{kernel} with $\alpha \in (0 \vee (d-2), d)$ is already discussed in \cite{Due16, Ser17,Ser20} based on the extension representation for the fractional Laplacian \cite{CaS07}. To be more specific, we notice that the convolution term $K \star \rho$ can be extended to be defined in $\R^d \times \R$ via
\[
\intr K ((x,\xi) - (y,0)) \rho(y)\,dy =: (K \star (\rho \otimes \delta_0))(x,\xi),
\]
where we denote
\[
K(x,\xi) = \frac{1}{|(x,\xi)|^\alpha}.
\]
Then taking $\gamma := \alpha + 1 - d \in (-1,1)$ yields that the extended interaction potential satisfies
\[
-\nabla_{(x,\xi)} \cdot \lt(\xi^\gamma \nabla_{(x,\xi)} K \star (\rho \otimes \delta_0) \rt) = c_{\alpha,d} \rho(x) \otimes \delta_0(\xi) \quad \mbox{on} \quad \R^d \times \R_+
\]
in the sense of distributions. Here $c_{\alpha,d}$ is the normalization constant explicitly given by
\[
c_{\alpha,d} = 2\alpha (2\pi)^{d/2} \Gamma\lt(\frac{\alpha+2-d}{2}\rt) \Gamma\lt(\frac{\alpha+2}{2}\rt)^{-1}
\]
for $\alpha \in (0 \vee (d-2), d)$. In particular, $K$ can be considered as a solution of 
\[
-\nabla_{(x,\xi)} \cdot \lt(\xi^\gamma  \nabla_{(x,\xi)} K \rt) = c_{\alpha,d} \delta_0 \quad \mbox{on} \quad \R^d \times \R_+.
\]
This interpretation plays a crucial role in the modulated interaction energy estimate. 

On the other hand, it is not clear how to apply the above argument for the mildly singular case than the Riesz one, i.e., $\alpha \in (0,d-2)$ with $d > 2$. The intuition for our main strategy stems from the following simple observation. Let us consider the interaction potential $K$ given by the fundamental solution to the biharmonic equation, that is, $K$ satisfies the following equation in the sense of distributions:
\bq\label{eq_bih}
(-\Delta)^2 K = \delta_0.
\eq
Then it can also be represented as
\[
-\Delta K = E, \quad -\Delta E = \delta_0.
\]
This shows that $E$ is the fundamental solution to the Laplace's equation, and thus,
\[
E(x) = \begin{cases}\displaystyle \frac{1}{(d-2)\omega_{d-1}} |x|^{2-d} \hspace{1cm} d\ge 3, \\[3mm]
\displaystyle -\frac{1}{2\pi} \log|x| \hspace{2.1cm}  d=2,  \end{cases}
\]
where $\omega_{d-1}$ denotes the area of the surface of the unit sphere $\mathbb{S}^{d-1}$ in $\R^d$. We then use the above observation to have that  $K$ can be represented as
\[
K(x) = \left\{\begin{array}{ll}\displaystyle \frac{1}{8\pi} |x|^2 \left(\log|x|-1\right) & d=2,\\[3mm]
\displaystyle-\frac{1}{4\omega_3}\log|x| & d=4,\\[3mm]
\displaystyle \frac{1}{2(d-2)(d-4)\omega_{d-1}}|x|^{4-d}, & d\neq 2,4.  \end{array}\right.
\]
This illustrates that the interaction potential $K(x) = |x|^{-(d-4)}$ can be interpreted as the solution to the biharmonic equation \eqref{eq_bih} up to a constant, and similarly as in the Coulomb potential case, this can be used for the estimate of modulated interaction energy when $\alpha = d-4$ with $d > 4$. From this simple observation, we expect that the potential $K$ of the form 

\[
K(x) = |x|^{-m}  \mbox{ with $m \in \N$}, \quad m<d  \quad \mbox{and}  \quad \mbox{$m\equiv d$ mod 2}
\]
can be interpreted as the solution to

\[
(-\Delta)^{\frac{d-m}{2}} K(x) = c_{m,d}\delta_0(x)
\]
in the sense of distributions, for some positive constant $c_{m,d}$. 
These eventually conclude that we can consider the kernel $K$ with $\alpha \in (0, d-2) \cap (d-2\N)$ as the solution to the polyharmonic equation. 

Let us now focus on the case $\alpha \in (0, d-2) \setminus \lt(d-2\N \rt)$ with $d>2$. For simple presentation of the idea, at the moment, we take into account the case $\alpha \in (0 \vee (d-4), d-2)$, inspired by the extension representation for the fractional Laplacian \cite{CaS07, Ypre}, we observe that $K = |x|^{-\alpha}$ with $\alpha \in (0 \vee (d-4), d-2)$ satisfies
 \[
 \nabla_{(x,\xi)} \cdot \lt( \xi^\gamma \nabla_{(x,\xi)} \lt( \frac{1}{\xi^\gamma} \nabla_{(x,\xi)}\cdot (\xi^\gamma \nabla_{(x,\xi)}K\star(\rho\otimes \delta_0)\rt)\rt) = c_{\alpha,d}\rho(x)\otimes\delta_0(\xi), 
 \]
 where $\gamma := \alpha+3-d \in (-1,1)$. In order to write the above in a compact form, we define an elliptic operator $\Delta_\gamma$ given as
 \bq\label{def_dg}
 \Delta_\gamma := \Delta_{(x,\xi)} + \frac{\gamma}{\xi}\pa_\xi.
 \eq
 Then we can easily check that the operator $\Delta_\gamma$ has the following relation:
 \bq\label{dg}
 \xi^\gamma \Delta_\gamma = \nabla_{(x,\xi)}\cdot(\xi^\gamma \nabla_{(x,\xi)}),
 \eq
 and thus we have
 \bq\label{dg2}
 \xi^\gamma \Delta_\gamma^2 K\star(\rho\otimes\delta_0)(x,\xi) = c_{\alpha,d} \rho(x)\otimes \delta(\xi)  \quad \mbox{in} \quad \R^d \times \R_+.
 \eq
This interpretation plays a significant role in estimating the modulated interaction energy. 

After having those observations, the proof follows from commutator estimates. To give the essential idea of our strategy, for an example, let us consider the Coulomb case, i.e., $\alpha = d-2$ with $d > 2$ and thus $-\Delta K = \delta_0$ up to a constant. Then we estimate
\begin{align*}
&\intr (\rho-\bar\rho) u \cdot \nabla K \star(\rho-\bar\rho)\,dx\cr
&\quad = -\intr \Delta K \star (\rho - \bar\rho) u \cdot \nabla K \star(\rho-\bar\rho)\,dx\cr
&\quad = \intr \nabla K \star (\rho - \bar\rho) \cdot \nabla (u \cdot \nabla K \star(\rho-\bar\rho))\,dx\cr
&\quad = \intr \nabla K \star (\rho - \bar\rho) \cdot \nabla^2 K \star (\rho - \bar\rho) u\,dx + \intr \nabla K \star (\rho - \bar\rho) [\nabla, u\cdot \nabla] K \star (\rho - \bar\rho)\,dx\cr
&\quad =: \sfI + \sfJ,
\end{align*}
where $[\cdot, \cdot]$ stands for the commutator operator, i.e., $[A,B] = AB - BA$. By the integration by parts, we get
\[
\sfI = -\frac12 \intr (\nabla \cdot u)|\nabla K \star (\rho - \bar\rho)|^2\,dx \leq \frac{\|\nabla u\|_{L^\infty}}{2}\intr |\nabla K \star (\rho - \bar\rho)|^2\,dx. 
\]
For $\sfJ$, by H\"older's inequality, we obtain
\[
\sfJ \leq \|\nabla K \star (\rho - \bar\rho)\|_{L^2}\|[\nabla, u \cdot \nabla] K \star (\rho - \bar\rho) \|_{L^2},
\]
and in this simple example, the term with commutator operator can be easily estimated as
\[
\|[\nabla, u \cdot \nabla] K \star (\rho - \bar\rho) \|_{L^2} \leq \|\nabla u\|_{L^\infty} \|\nabla K \star (\rho - \bar\rho)\|_{L^2}.
\]
On the other hand, by using the fact that $K$ satisfies $-\Delta K = \delta_0$, we can readily check that
\[
\|\nabla K \star (\rho - \bar\rho)\|_{L^2}^2 = \intr (\rho - \bar\rho) K \star (\rho - \bar\rho)\,dx.
\]
This shows that the Coulomb case satisfies \eqref{ess}, and thus we have \eqref{res_thm} with $C>0$ which depends only on $\|\nabla u\|_{L^\infty}$. We extend this argument to cover the regime $(0,d)$. As mentioned before, these types of estimates are well developed in \cite{Due16,Ser17,Ser20} for $\alpha \in (0\vee (d-2),d)$. Thus, we will only focus on the case $\alpha \in (0,d-2)$.

\subsection{Applications: quantified asymptotic analysis for kinetic equations}

\subsubsection{Small inertia limit of kinetic equation} 
As applications of Theorem \ref{thm_main}, we first study the asymptotic behavior of the following kinetic equation under the small inertia regime:
\bq\label{main_kin0}
\e\pa_t f^\e + \e v \cdot \nabla_x f^\e - \nabla_v \cdot \lt((\gamma v   + \nabla K \star \rho^\e )f^\e\rt) =0, \quad (x,v) \in \R^d \times \R^d,  \quad t > 0,
\eq
subject to the initial data:
\[
f^\e(x,v,t)|_{t=0} =: f_0^\e(x,v), \quad (x,v) \in \R^d \times \R^d,
\]
where $K$ is the interaction potential, which is of the form  $K = |x|^{-\alpha}$ with $\alpha \in (0,d)$, and $\gamma > 0$ denotes the strength of linear damping in velocity. Here $\rho^\e = \rho^\e(x,t)$ is the local particle density:
\[
\rho^\e = \intr f^\e\,dv.
\]
For the kinetic equation \eqref{main_kin0}, we are interested in the behavior of solutions $f^\e$ when $\e \to 0$. 

At the formal level, it follows from \eqref{main_kin0} that 
\[
(\gamma v   + \nabla K \star \rho^\e )f^\e \simeq 0 \quad \mbox{for} \quad \e \ll 1,
\]
and thus, this deduces that $f^\e$ becomes the mono-kinetic distribution of the form:
\[
f^\e(x,v) \simeq \rho^\e(x) \otimes \delta_{u^\e}(v) \quad \mbox{with} \quad \gamma u^\e =   - \nabla K \star \rho^\e.
\]
Thus, once we get $\rho^\e \to \rho$ and $u^\e \to u$ as $\e \to 0$, by estimating local moment in velocity on $f$, we find 
\bq\label{eq_agg}
\pa_t \rho + \nabla_x \cdot (\rho u) = 0, \quad \gamma  \rho u =  -    \rho \nabla K \star \rho.
\eq
The asymptotic analysis of the equation \eqref{main_kin0} with regular interaction potentials $K$ when $\e \to 0$ is discussed in \cite{FS15, FST16, Jab00} based on compactness arguments. More recently, a quantitative estimate for the small inertia limit of kinetic equation is obtained in \cite{CC20}. In \cite{CC20}, the combination of the modulated macro-kinetic energy and $2$-Wasserstein distance plays a crucial role in having that quantitative error estimate between local densities. However, in those works, the strong regularity for the interaction potential $K$ is required, and it is totally unclear for the singular potentials. To the best of our knowledge, the small inertia limit for the kinetic equation \eqref{main_kin0} with singular kernels has not been studied so far. 

In order to overcome difficulties arising from the singular interaction potentials, we employ the modulated meso-kinetic energy combined with the modulated interaction energy estimate in Theorem \ref{thm_main} to obtain the quantitative error bounds between \eqref{main_kin0} and \eqref{eq_agg}.

\begin{theorem}\label{thm_ktoc} Let $T>0$ and $d \geq 1$. Let $f^\e \in \mc([0,T); \mathcal{P}(\R^d \times \R^d))$ be a solution to the equation \eqref{main_kin0} in the sense of distributions, and let $(\rho, u)$ be the unique classical solution of the aggregation-type equation \eqref{eq_agg} with $\rho > 0$ on $\R^d \times [0,T)$, $u \in W^{1,\infty}(\R^d \times (0,T))$, and if $\alpha<d-2$, $\nabla^{[(d-\alpha)/2]+1} u\in L^\infty(0,T;L^{\frac{d}{[(d-\alpha)/2]}}(\R^d))$ up to time $T>0$ with the initial data $ \rho_0$. If 
\[
\sup_{\e > 0}\intrr |v - u_0(x) |^2 f^\e_0(x,v)\,dxdv < \infty
\]
and
\[
\intr (\rho_0 - \rho^\e_0)K\star(\rho_0 - \rho^\e_0)\,dx + \d_{BL}(\rho_0, \rho_0^\e) \to 0 \quad \mbox{as} \quad \e \to 0,
\]
then we have
\begin{align*}
\intr \,f^\e\,dv &\rightharpoonup \rho \quad \mbox{weakly in } L^\infty(0,T;\mathcal{M}(\R^d)), \cr
\intr v \,f^\e\,dv  &\rightharpoonup \rho u  \quad \mbox{weakly in } L^2(0,T;\mathcal{M}(\R^d)), 
\end{align*}
and
\[
f^\e \rightharpoonup \rho\delta_{u} \quad \mbox{weakly in } L^2(0,T;\mathcal{M}(\R^d \times \R^d)),
\]
where $\d_{BL}$ stands for the bounded-Lipschitz distance, and we denoted by $\mathcal{M}(\R^n)$ the space of signed Radon measures on $\R^n$ with $n \in \N$.
\end{theorem}

\begin{remark} From the convergence of modulated interaction energy, we obtain 
\[
\rho^\e \to \rho \quad \mbox{in } L^\infty(0,T; \dot{H}^{-\frac{d-\alpha}{2}}(\R^d)).
\]
\end{remark}

\begin{remark} As stated in Theorem \ref{thm_ktoc}, we need the unique classical solution to the limiting equation \eqref{eq_agg} to make all our estimates completely rigorous. For this, we refer to \cite{CJ21}, where the local-in-time existence and uniqueness of classical solutions are established. On the other hand, for the kinetic equation, we need a rather weak regularity of solutions. When $\alpha \in (0,d-1)$, the global-in-time existence theory for weak solutions obtained in \cite{CCJ21} would be adopted, see also \cite{CCS19} for the case $\alpha \in (0,d-2)$. For the case $\alpha \in [d-1,d)$, up to our limited knowledge, the global-in-time existence theory is not available up to now; the local-in-time existence and uniqueness of strong solutions are investigated in \cite{IVDB98} for $\alpha = d-1$. For that reason, our result on the small inertia limit is rigorous at the formal level when $\alpha \in (d-1,d)$.
\end{remark}

\begin{remark} In order to derive the equation  \eqref{eq_agg} with a linear diffusion, i.e. aggregation-diffusion equation, the Vlasov--Fokker--Planck equation can be considered. More precisely, in \cite{CTpre}, the equation \eqref{main_kin0} with the linear diffusion term $\gamma \Delta_v f$ on the right hand side, and quantified asymptotic analysis under the overdamped regime is discussed, see also \cite{CCPpre, GM10, PS00}. 
\end{remark}

\begin{remark} The fractional porous medium equation \cite{ CV11, CV11_2, CJ21}
\[
\pa_t \rho + \nabla \cdot (\rho \nabla K\star \rho) = c_P \Delta \rho^\gamma
\]
can be rigorously and quantitatively derived from the following damped Euler equations as $\e \to 0$:
\[
\begin{aligned}
&\pa_t \rho^\e + \nabla \cdot (\rho^\e u^\e) = 0, \\
&\e\pa_t (\rho^\e u^\e) + \e\nabla \cdot ( \rho^\e u^\e\otimes u^\e) +    c_P \nabla p(\rho^\e) =  -   \rho^\e u^\e -    \rho^\e \nabla K \star \rho^\e.
\end{aligned}
\]
Here $c_P \geq 0$ is the pressure coefficient and the pressure $p = p(\rho)$ is given by the power-law $p(\rho) = \rho^\delta$ for some $\delta \geq 1$. The Coulomb case is investigated for the strong relaxation limit from Euler equations to the aggregation equation in \cite{CPW20, C21, LT13, LT17}, see also \cite{CCT19}.  More recently, the Riesz interaction case $\alpha \in (d-2,d)$, and more generally,  the case $\alpha \in (0,d)$, is dealt with in \cite{CJpre1}, see \cite{CJpre3, CJpre} for the well-posedness theory for the Euler--Riesz system. The main strategy is based on the modulated energies, e.g., kinetic, internal, interaction energies, combined with the bounded-Lipschitz distance estimates. It is worth mentioning that it is typically assumed that the damping coefficient $\gamma >0$ is chosen sufficiently large for the relaxation limit from Euler equations towards aggregation equation, however, in our result, Theorem \ref{thm_ktoc}, we do not require such assumption on $\gamma$. 
\end{remark}

\subsubsection{Hydrodynamic limit from kinetic to Euler equations}
We next apply the modulated interaction energy estimate in Theorem \ref{thm_main} to study the hydrodynamic limit of the kinetic equation to the Euler-type system. We now consider the nonlinear Fokker--Planck operator in the kinetic equation \eqref{main_kin0}:
\begin{align}\label{main_eq}
\begin{aligned}
&\pa_t f + v\cdot\nabla_x f -  \nabla_v \cdot \lt((\gamma v +  \nabla K \star \rho )f \rt)  = \mathcal{N}_{FP}[f], \quad (x,v,t) \in \R^d \times \R^d \times \R_+,
\end{aligned}
\end{align}
where the nonlinear Fokker--Planck operator $\mathcal{N}_{FP}[f]$ is given by
\[
 \mathcal{N}_{FP}[f](x,v) := \nabla_v \cdot (\beta (v - u)f + \sigma \nabla_v f) = \sigma \nabla_v \cdot \lt(f \nabla_v \log \frac{f}{M_u} \rt) 
\]
with the local Maxwellian
\[
M_u := \frac{\beta^{d/2}}{(2\pi \sigma)^{d/2}} \exp\lt(-\frac{\beta |u-v|^2}{2\sigma} \rt) ,
\]
and positive constants $\beta$ and $\sigma$. Here the term with $\beta$ is the local velocity alignment force and the other term with $\sigma$ is the linear diffusion in velocity. 

Then we deal with two asymptotic regimes for \eqref{main_eq}: strong local alignment without diffusion ($\beta = \e^{-1}$ and $\sigma = 0$) and strong local alignment and diffusion ($\beta = \sigma = \e^{-1}$). More precisely, we will show that the equation \eqref{main_eq} converges towards the following Euler-type system as $\e \to 0$:
\begin{align}\label{main_E}
\begin{aligned}
&\pa_t \rho + \nabla_x \cdot (\rho u) = 0, \quad (x,t) \in \R^d \times \R_+,\cr
&\pa_t (\rho u) + \nabla_x \cdot (\rho u \otimes u) + c_P\nabla_x \rho = - \gamma \rho u -  \rho \nabla_x K \star \rho,
\end{aligned}
\end{align}
where the coefficient of pressure $c_P = c_P(\sigma) \geq 0$ is given by
\[
c_P = 0  \mbox{ when }  \sigma = 0 \quad \mbox{and} \quad c_P = 1 \mbox{ when $\sigma = \e^{-1}$}.
\]

The hydrodynamic limit from \eqref{main_eq} towards \eqref{main_E} is well studied in \cite{FK19, KMT15} when $K \equiv 0$. Recently, an unified approach for the hydrodynamic limits of \eqref{main_eq} with the interaction potential $K$, which has either the Coulomb singularity or bounded-Lipschitz continuity, in \cite{CCJ21}, regardless of the presence of the diffusion. The main idea of proof relies on the classical relative entropy argument combined with the modulated interaction energy and bounded-Lipschitz distance estimates. To extend this result to cover the other singular interaction potentials $K = |x|^{-\alpha}$ with $\alpha \in (0,d)$, we need to estimate the term with the potential $K$ differently from that of \cite{CCJ21}, but this can be easily handled by using Theorem \ref{thm_main}.

In order to state the main theorem in this part, we introduce the free energy $\mf$ and the dissipation $\md$ as follows: 
\[
\mf(f) := c_P\intrr f \log f\,dxdv + \frac12\intrr |v|^2 f\,dxdv + \frac{1}{2}\intrr K(x-y)\rho(x) \rho(y)\,dxdy 
\]
and
\[
\md(f) := \intrr \frac{1}{f} \lt|c_P \nabla_v f - f(u-v) \rt|^2 dxdv,
\]
respectively.

\begin{theorem}\label{thm_hydro} Let $T>0$ and $d \geq 1$. Let $f^\e$ be a weak solution to the equation \eqref{main_eq} with either $\beta = 1/\e$ and $\sigma = 0$ or $\beta = \sigma = 1/\e$ satisfying $f \in L^\infty(0,T;(L^1_+ \cap L^\infty)(\R^d \times \R^d))$ and the free energy inequality:
\[
\mf(f^\e) + \int_0^t \lt(\frac1{2\e}\md(f^\e) + \gamma \intr \rho^\e|u^\e|^2\,dx \rt)ds \leq \mf(f_0^\e) + C(1+ \gamma^2)\e,
\]
where $C > 0$ depends only $T$ and $f_0^\e$.
Let $(\rho,u)$ be the unique classical solution to the system \eqref{main_E}  satisfying $u \in W^{1,\infty}(\R^d \times (0,T))$, and if $\alpha<d-2$, $\nabla^{[(d-\alpha)/2]+1} u\in L^\infty(0,T;L^{\frac{d}{[(d-\alpha)/2]}}(\R^d))$ up to the time $T > 0$. Suppose that the initial data are well-prepared in the sense that the following assumptions hold:
\begin{itemize}
\item[(i)] 
$ 
\intr \rho_0^\e|u_0 - u^\e_0|^2\,dx = \mathcal{O}(\sqrt\e)$ and $  \intr \lt( \intr f_0^\e |v|^2\,dv  - \rho_0|u_0|^2\rt)dx  = \mathcal{O}(\sqrt\e), 
$
\item[(ii)] $\intr (\rho_0 - \rho_0^\e)K\star(\rho_0 - \rho_0^\e)\,dx = O(\sqrt\e), $
\item[(iii)] $d^2_{BL}(\rho^\e_0, \rho_0) = \mathcal{O}(\sqrt\e)$.
\end{itemize}
When $\sigma \neq 0$ $(c_P = 1)$, we further assume 
\[
\rho_0^\e\lt( \log \rho^\e_0 - \log \rho_0 \rt) + (\rho_0 - \rho_0^\e) = \mathcal{O}(\sqrt\e)\quad \mbox{and} \quad \intr\lt(\intr f_0^\e \log f_0^\e\,dv - \rho_0 \log \rho_0 \rt)dx = \mathcal{O}(\sqrt\e).
\]
If $\sigma = \beta = \e^{-1}$,  then we have
\bq\label{conv_res1}
\intr \,f^\e\,dv \to \rho, \quad \intr v \,f^\e\,dv  \to \rho u,    \quad \mbox{in } L^\infty(0,T;L^1(\R^d))
\eq
and
\[
f^\e \to \frac{\rho}{(2\pi)^{d/2}} e^{-\frac{|u-v|^2}{2}} \quad \mbox{in } L^\infty(0,T;L^1(\R^d \times \R^d)) \quad \mbox{as } \e \to 0. 
\]
On the other hand, if $\beta = 1/\e$ and $\sigma = 0$ $(c_P = 0)$, then we have the convergences \eqref{conv_res1} weakly in $L^\infty(0,T;\mathcal{M}(\R^d))$ and 
\[
f^\e \rightharpoonup \rho \otimes \delta_u \quad \mbox{weakly in } L^\infty(0,T;\mathcal{M}(\R^d \times \R^d)) \quad \mbox{as } \e \to 0.
\]
\end{theorem}

\begin{remark} If we choose the initial data $f_0^\e$ as
\[
f_0^\e(x,v) = \frac{\rho_0(x)}{(2\pi)^{d/2}}\exp\lt(-\frac{|u_0(x) - v|^2}{2}\rt) \quad \mbox{for all} \quad \e > 0,
\]
then we obtain
\[
\rho^\e_0 = \intr f^\e_0\,dv = \rho_0 \quad \mbox{and} \quad \rho^\e_0 u^\e_0 = \intr v f^\e_0\,dv = \intr u_0 f^\e_0\,dv =\rho_0 u_0.
\]
\end{remark}

\begin{remark} For the existence and uniqueness of classical solutions for the limiting equation \eqref{main_E}, we refer to \cite{CJpre3, CJpre} for the case $\alpha \in (d-2,d)$; local-in-time well-posedness theory is established in \cite{CJpre3} even for the attractive Riesz interaction case, and the global-in-time theory is obtained in \cite{CJpre} for the system \eqref{main_E} with $c_P = 0$. For the case $\alpha \in (0,d-2]$, a simple modification of the strategy used in \cite{CCJ21} would be applied to obtain the local-in-time well-posedness for \eqref{main_E}.
\end{remark}

\begin{remark}The assumptions on the regularity of solutions can be relaxed. One may check the argument proposed in \cite{LT17} to estimate our relative entropy appearing in Section \ref{ssec_hydro} by taking into account the following notation of solutions $(\rho,u)$ to \eqref{main_E}:
\begin{itemize}
\item[(i)] $(\rho, u) \in \mc([0,T];L^1(\R^d)) \times \mc([0,T];W^{1,\infty}(\R^d))$,
\item[(ii)] $(\rho, u)$ satisfies the following free energy estimate in the sense of distributions:
$$\begin{aligned}
&\frac{d}{dt}\lt(\frac12\intr \rho|u|^2\,dx  + c_P\intr \rho \log \rho\,dx   + \frac12\intr (K \star \rho)\rho\,dx \rt) = - \gamma \intr \rho|u|^2\,dx,
\end{aligned}$$
\item[(iii)] $(\rho, u)$ satisfies the system \eqref{main_E} in the sense of distributions.
\end{itemize}
\end{remark}

The rest of this paper is organized as follows. We provide the details of the proof for our main result on the modulated interaction energy in Theorem \ref{thm_main} by dealing with two cases: $\alpha \in (0,d-2) \setminus(d-2 \N)$ and $\alpha \in (0,d-2) \cap (d-2\N)$. In Section \ref{sec_subCoul}, we consider the former case by introducing the elliptic operator $\Delta_\gamma$ and applying the dimension extension argument. Section \ref{sec_subCoul2} is devoted to take into account the other case. Finally, in Section \ref{sec:appl}, we prove Theorems \ref{thm_ktoc} and \ref{thm_hydro} on the quantified small inertia limit and hydrodynamic limit of kinetic equations.

%
%
%
%
%
%

\section{Modulated interaction energy estimates: $\alpha \in (0,d-2) \setminus (d-2 \N)$.}\label{sec_subCoul}

The objective of this section is to provide the details on the modulated interaction energy estimate in Theorem \ref{thm_main} when $\alpha \in (0,d-2) \setminus (d-2\N)$. In order to introduce a general strategy of the proof, in the following two subsections, we consider the exponent $\alpha$ satisfying $\alpha \in (0 \vee (d-4), d-2) $ with describing the main ideas behind our strategy for the general case.

%
%
%
%
%
%
\subsection{Case: $0 \vee (d-4) < \alpha<d-2$}
In this subsection, we consider the interaction potential $K$ given by \eqref{kernel} with $0 \vee (d-4)< \alpha<d-2$ and $d>2$.

Let $\alpha \in (0 \vee (d-4), d-2)$. 
 Straightforward computation shows that we can interpret the interaction potential $K$ satisfies
 \[
 \nabla_{(x,\xi)} \cdot \lt( \xi^\gamma \nabla_{(x,\xi)} \lt( \frac{1}{\xi^\gamma} \nabla_{(x,\xi)}\cdot (\xi^\gamma \nabla_{(x,\xi)}K\star(\rho\otimes \delta_0)\rt)\rt) = c_{\alpha,d}\rho(x)\otimes\delta_0(\xi), 
 \]
 where $\gamma := \alpha+3-d \in (-1,1)$. In order to write the above in a compact form, we define an elliptic operator $\Delta_\gamma$ given as
 \bq\label{def_dg}
 \Delta_\gamma := \Delta_{(x,\xi)} + \frac{\gamma}{\xi}\pa_\xi.
 \eq
 Then we can easily check that the operator $\Delta_\gamma$ has the following relation:
 \bq\label{dg}
 \xi^\gamma \Delta_\gamma = \nabla_{(x,\xi)}\cdot(\xi^\gamma \nabla_{(x,\xi)})=\xi^\gamma \Delta_{(x,\xi)} + \pa_\xi (\xi^\gamma)\pa_\xi,
 \eq
 and thus we have
 \bq\label{dg2}
 \xi^\gamma \Delta_\gamma^2 K\star(\rho\otimes\delta_0)(x,\xi) = c_{\alpha,d} \rho(x)\otimes \delta(\xi)  \quad \mbox{in} \quad \R^d \times \R_+.
 \eq
It is clear that the operators $\Delta_\gamma$ and $\nabla_x$ commute, however, $\Delta_\gamma$ and $\nabla_{(x,\xi)}$ do not commute, i.e. $\Delta_\gamma \nabla_{(x,\xi)} \neq \nabla_{(x,\xi)} \Delta_\gamma$. Indeed, for $f = f(x,\xi) \in \mc^3(\R^d \times \R_+)$, we find
\[
(\nabla_{(x,\xi)} \Delta_\gamma - \Delta_\gamma \nabla_{(x,\xi)})f = (0, [\pa_\xi, \gamma\xi^{-1}]\pa_\xi f) = (0, -\gamma \xi^{-2}\pa_\xi f).
\]
This implies the operators $\Delta_\gamma$ and $\nabla_{(x,\xi)}$ do commute except the $(d+1)$-th component. On the other hand, this observation gives
\[
({\bf a},0) \cdot (\nabla_{(x,\xi)} \Delta_\gamma - \Delta_\gamma \nabla_{(x,\xi)})f = 0
\]
for any ${\bf a} \in \R^d$.

Before proceeding further, we present some technical estimates below. For simplicity of presentation, we assume that the positive constant $c_{\alpha,d} = 1$ in the following computations.
 \begin{lemma}\label{lem_dg}
 \begin{itemize}
 \item[(i)] For $f,g \in \mc^2(\R^d \times \R_+)$, we obtain
 \[
 \iint_{\R^d \times \R_+} \xi^\gamma (\Delta_\gamma f) g\,dxd\xi =  \iint_{\R^d \times \R_+} \xi^\gamma f(\Delta_\gamma g) \,dxd\xi.
 \]
  \item[(ii)]   The modulated interaction energy can be rewritten as 
  \[
  \iint_{\R^d\times\R_+} \xi^\gamma | \Delta_\gamma K\star((\rho-\bar\rho)\otimes \delta_0))|^2\,dxd\xi.
  \]
 \end{itemize}
 \end{lemma}
 \begin{proof} (i) It follows from \eqref{dg} that 
 \begin{align*}
\iint_{\R^d \times \R_+} \xi^\gamma (\Delta_\gamma f) g\,dxd\xi &=  \iint_{\R^d \times \R_+} \nabla_{(x,\xi)}\cdot(\xi^\gamma \nabla_{(x,\xi)}f) g\,dxd\xi \cr
&= -\iint_{\R^d \times \R_+} \xi^\gamma \nabla_{(x,\xi)}f \cdot \nabla_{(x,\xi)}g\,dxd\xi \cr
&=  \iint_{\R^d \times \R_+} f\nabla_{(x,\xi)}\cdot(\xi^\gamma \nabla_{(x,\xi)}g) \,dxd\xi \cr
&=  \iint_{\R^d \times \R_+} \xi^\gamma f(\Delta_\gamma g) \,dxd\xi.
 \end{align*}
 (ii) Note that
 \begin{align*}
 \intr (\rho -\bar \rho) K \star (\rho - \bar\rho)\,dx &= \iint_{\R^d \times \R_+} \lt((\rho -\bar \rho)(x) \otimes \delta_0(\xi)
\rt) K \star ((\bar\rho - \rho) \otimes \delta_0)(x,\xi)\,dxd\xi\cr
&=\iint_{\R^d \times \R_+} \lt(\xi^\gamma \Delta_\gamma^2 K\star(\rho\otimes\delta_0)(x,\xi) 
\rt) K \star ((\bar\rho - \rho) \otimes \delta_0)(x,\xi)\,dxd\xi,
 \end{align*}
due to \eqref{dg2}. We finally use (i) to conclude the desired result.
 \end{proof}
 
We notice that the elliptic operator $\Delta_\gamma$ consists of the usual Laplace operator defined on the upper half space and the differential operator with respect to the extension variable. As observed in Lemma \ref{lem_dg} (i), this operator behaves under integration by parts in the weighted Sobolev space as the usual Laplacian does in the unweighted one. 

Our next concern is to find a relation between the $L^2$ norm of $\Delta_\gamma K\star((\rho-\bar\rho)\otimes \delta_0)$ and $\nabla_{(x,\xi)}^2 K\star((\rho-\bar\rho)\otimes \delta_0)$ in the weighted space. In fact, the following lower bound estimate of the modulated energy plays a crucial role in our main strategy. For notational simplicity, throughout this subsection, we often write
\[
\K := K\star((\rho-\bar\rho)\otimes \delta_0).
\]
\begin{lemma}\label{lem_dg2} Suppose that the exponent $\alpha$ satisfies $0 \vee (d-4)<\alpha<d-2$. Then we have
\bq\label{id_del}
\iint_{\R^d\times\R_+}\xi^\gamma |\nabla_{(x,\xi)}^2 \K|^2\,dxd\xi =(1-\gamma/4) \iint_{\R^d\times\R_+} \xi^\gamma |\Delta_\gamma \K|^2\,dxd\xi,
\eq
where $\gamma :=   \alpha+3-d \in (-1,1)$. 
\end{lemma} 
\begin{proof} We begin with the estimate of the first term on the right hand side of \eqref{id_del}. 
\begin{align*}
&\iint_{\R^d\times\R_+}\xi^\gamma \nabla_{(x,\xi)}^2  \K  :  \nabla_{(x,\xi)}^2  \K\,dxd\xi\\
&\quad = -\iint_{\R^d\times\R_+} \nabla_{(x,\xi)}\cdot \lt(\xi^\gamma \nabla_{(x,\xi)}^2  \K  \rt) \cdot \nabla_{(x,\xi)}  \K\,dxd\xi\\
&\quad = -\iint_{\R^d\times\R_+} \xi^\gamma \Delta_\gamma \lt(\nabla_{(x,\xi)} \K\rt) \cdot \nabla_{(x,\xi)}  \K\,dxd\xi\\
&\quad = -\iint_{\R^d\times\R_+} \xi^\gamma \Delta_{(x,\xi)} \nabla_{(x,\xi)} \K \cdot \nabla_{(x,\xi)} \K\,dxd\xi - \iint_{\R^d\times\R_+} \pa_\xi(\xi^{\gamma})\pa_\xi (\nabla_{(x,\xi)}\K ) \cdot \nabla_{(x,\xi)} \K\,dxd\xi\\
&\quad =: \sfI_1 + \sfI_2,
\end{align*}
due to \eqref{dg}. Here, we further estimate $\sfI_1$ as
\begin{align*}
\sfI_1&= \iint_{\R^d\times\R_+} \Delta_{(x,\xi)} \K \nabla_{(x,\xi)}\cdot(\xi^\gamma \nabla_{(x,\xi)}\K)\,dxd\xi\\
&=  \iint_{\R^d\times\R_+} \xi^\gamma (\Delta_{(x,\xi)} \K) (\Delta_\gamma \K)\,dxd\xi\\
&= \iint_{\R^d\times\R_+} \xi^\gamma | \Delta_\gamma \K|^2\,dxd\xi -\gamma\iint_{\R^d\times\R_+} \xi^{\gamma-1} (\pa_\xi \K) (\Delta_\gamma \K)\,dxd\xi\\
&= \iint_{\R^d\times\R_+} \xi^\gamma | \Delta_\gamma \K|^2\,dxd\xi  -\iint_{\R^d\times\R_+}\pa_\xi( \xi^{\gamma}) (\pa_\xi \K) (\Delta_\gamma \K)\,dxd\xi.
\end{align*}
For $\sfI_2$, by the integration by parts, we deduce
\begin{align*}
\sfI_2 &= \iint_{\R^d\times\R_+} (\pa_\xi \K) \nabla_{(x,\xi)}\cdot \lt(\pa_\xi(\xi^\gamma) \nabla_{(x,\xi)} \K \rt)\,dxd\xi\\
&=  \iint_{\R^d\times\R_+} (\pa_\xi \K) \pa_\xi(\xi^\gamma) (\Delta_{(x,\xi)} \K )\,dxd\xi +  \iint_{\R^d\times\R_+} \pa_\xi^2 (\xi^\gamma)  |\pa_\xi \K |^2\,dxd\xi\\
&= \iint_{\R^d\times\R_+} (\pa_\xi \K) \pa_\xi(\xi^\gamma) (\Delta_{\gamma} \K )\,dxd\xi -\iint_{\R^d\times\R_+} \frac{\gamma}{\xi}\pa_\xi (\xi^\gamma)  |\pa_\xi \K |^2\,dxd\xi\\
&\quad +  \iint_{\R^d\times\R_+} \pa_\xi^2 (\xi^\gamma)  |\pa_\xi \K |^2\,dxd\xi\cr
&= \iint_{\R^d\times\R_+} (\pa_\xi \K) \pa_\xi(\xi^\gamma) (\Delta_{\gamma} \K )\,dxd\xi -\gamma \iint_{\R^d\times\R_+} \xi^{\gamma-2} |\pa_\xi \K |^2\,dxd\xi.
\end{align*}
We then combine the above two estimates to find
\[
\sfI_1 + \sfI_2= \iint_{\R^d\times\R_+} \xi^\gamma | \Delta_\gamma \K|^2\,dxd\xi -\gamma \iint_{\R^d\times\R_+} \xi^{\gamma-2} |\pa_\xi \K |^2\,dxd\xi.
\]
Finally, observe that
\[
\Delta_\gamma \K = -2\alpha \int_{\R^d} \frac{(\rho-\bar\rho)(y)}{|(x-y, \xi)|^{\alpha+2}}\,dy \quad \mbox{and} \quad \frac{\pa_\xi}{\xi}\K = -\alpha  \int_{\R^d} \frac{(\rho-\bar\rho)(y)}{|(x-y, \xi)|^{\alpha+2}}\,dy = \frac12\Delta_\gamma \K,
\]
which gives
\[
\iint_{\R^d\times\R_+} \xi^{\gamma-2} |\pa_\xi\K|^2 \,dxd\xi = \iint_{\R^d\times\R_+} \xi^{\gamma} \lt|\frac{\pa_\xi}{\xi}\K\rt|^2 \,dxd\xi = \frac14\iint_{\R^d\times\R_+} \xi^{\gamma} |\Delta_\gamma \K|^2 \,dxd\xi.
\]
This completes the proof.
\end{proof}

\begin{proof}[Proof of Theorem \ref{thm_main} when $0 \vee (d-4)<\alpha<d-2$]  
By using \eqref{dg2} and Lemma \ref{lem_dg} (i), we first find  
\[\begin{aligned}
\int_{\R^d}& (\rho-\bar\rho) u \cdot \nabla K \star(\rho-\bar\rho)\,dx\\
&=\iint_{\R^d\times\R_+} \xi^\gamma (\Delta_\gamma^2  \K )   \lt(\nabla_{(x,\xi)} \K \cdot (u(x),0)\rt)\,dxd\xi\\
&= \iint_{\R^d\times\R_+} \xi^\gamma (\Delta_\gamma  \K) \, \Delta_\gamma \lt(\nabla_{(x,\xi)} \K \cdot (u(x),0)\rt)\,dxd\xi\\
&= \frac12\iint_{\R^d\times\R_+} \nabla_{(x,\xi)}|\Delta_\gamma  \K|^2 \cdot (u(x)\xi^\gamma, 0)\,dxd\xi + \iint_{\R^d\times\R_+} \xi^\gamma (\Delta_\gamma  \K) [\Delta_\gamma, (u(x),0)] \cdot \nabla_{(x,\xi)}  \K\,dxd\xi\\
&=  -\frac12\iint_{\R^d\times\R_+}(\nabla_x \cdot u)\xi^\gamma |\Delta_\gamma  \K|^2 \,dxd\xi + \iint_{\R^d\times\R_+} \xi^\gamma (\Delta_\gamma  \K) [\Delta_\gamma, (u(x),0)] \cdot \nabla_{(x,\xi)} \K\,dxd\xi,
\end{aligned}\]
where the first term on the right hand side can be easily bounded from above by
\[
\|\nabla_x \cdot u\|_{L^\infty}\iint_{\R^d\times\R_+}\xi^\gamma |\Delta_\gamma  \K|^2 \,dxd\xi. 
\]
On the other hand, the term with commutator can be estimated as
\begin{align*}
[ \Delta_\gamma, (u(x),0) ]\cdot\nabla_{(x,\xi)} \K 
&=[\Delta_{(x,\xi)}, (u(x),0)]\cdot\nabla_{(x,\xi)} \K \cr
&= (\Delta_x u(x),0) \cdot \nabla_{(x,\xi)} \K + 2 \nabla_{(x,\xi)}(u(x),0) : \nabla_{(x,\xi)}^2 \K\cr
&= \Delta_x u \cdot \nabla_x \K + 2 \nabla_x u : \nabla_x^2 \K,
\end{align*}
where we used
\[
[\Delta_\gamma, f(x)] g = [\Delta_{(x,\xi)}, f(x)]g
\]
for any $f: \R^d \to \R^d \times \R$ and $g: \R^d \times \R_+ \to \R^d \times \R$ with $f \in \mc^2(\R^d)$ and $g \in \mc^2(\R^d \times \R_+)$.

This implies
\[\begin{aligned}
& \iint_{\R^d\times\R_+} \xi^\gamma (\Delta_\gamma  \K) [\Delta_\gamma, (u(x),0)] \cdot \nabla_{(x,\xi)} \K\,dxd\xi\\
&\quad =  \iint_{\R^d\times\R_+}\xi^\gamma (\Delta_\gamma  \K)  (\Delta_x u \cdot \nabla_x \K) \,dxd\xi + 2 \iint_{\R^d\times\R_+}\xi^\gamma (\Delta_\gamma  \K)  \nabla_x u : \nabla_x^2 \K\,dxd\xi\\
&\quad =: \sfJ_1 + \sfJ_2.
\end{aligned}\]
For $\sfJ_1$, we use H\"older's inequality to get
\[\begin{aligned}
\sfJ_1 &\le \int_{\R_+} \xi^\gamma \|\Delta_\gamma \K\|_{L_x^2} \|\Delta_x u \cdot \nabla_x \K\|_{L^2}\,d\xi\\
&\le C\int_{\R_+} \xi^\gamma  \|\Delta_\gamma \K\|_{L_x^2} \|\Delta u\|_{L^d} \| \nabla_x \K\|_{L_x^{\frac{2d}{d-2}}}\,d\xi\\
&\le C \int_{\R_+} \xi^\gamma  \|\Delta_\gamma \K\|_{L_x^2} \|\nabla_x^2 \K\|_{L_x^2}\,d\xi\\
&\le C\|\xi^\gamma \Delta_\gamma \K\|_{L^2} \|\xi^\gamma \nabla_x^2 \K\|_{L^2},
\end{aligned}\]
where $C > 0$ depends on $\|\Delta u\|_{L^\infty(0,T;L^d)}$. 

For $\sfJ_2$, we have
\[\begin{aligned}
\sfJ_2&\le2 \|\xi^{\gamma/2} \Delta_\gamma \K\|_{L^2} \lt(\int_{\R_+} \xi^\gamma \lt( \int_{\R^d} |\nabla_x u|^2 |\nabla_x^2 \K|^2 \,dx\rt)d\xi \rt)^{1/2}\\
&\le C \|\xi^{\gamma/2} \Delta_\gamma \K\|_{L^2} \lt(\iint_{\R^d\times\R_+} \xi^\gamma  |\nabla_x^2 \K|^2 \,dxd\xi \rt)^{1/2}.
\end{aligned}\]
Combining all of the above observations implies
\[
\int_{\R^d} (\rho-\bar\rho) u \cdot \nabla K \star(\rho-\bar\rho)\,dx \leq C\|\xi^{\gamma/2} \Delta_\gamma \K\|_{L^2}\lt(\|\xi^{\gamma/2} \Delta_\gamma \K\|_{L^2} + \|\xi^{\gamma/2} \nabla_x^2 \K\|_{L^2}\rt).
\]
Since $|\gamma|\le 1$, we use Lemma \ref{lem_dg2} to have
\[
\|\xi^{\gamma/2} \nabla_x^2 \K\|_{L^2} \leq \|\xi^{\gamma/2} \nabla_{(x,\xi)}^2 \K\|_{L^2} \leq 2\|\xi^{\gamma/2} \Delta_\gamma \K\|_{L^2}.
\]
This together with Lemma \ref{lem_dg} (ii) yields
\[
\begin{aligned}
\int_{\R^d} (\rho-\bar\rho) u \cdot \nabla K \star(\rho-\bar\rho)\,dx &\leq C  \iint_{\R^d\times\R_+} \xi^\gamma | \Delta_\gamma K\star((\rho-\bar\rho)\otimes \delta_0))|^2\,dxd\xi\cr
&=C\intr (\rho -\bar\rho) K \star (\rho - \bar\rho)\,dx,
\end{aligned}
\]
where $C>0$ depends on $\|\nabla u\|_{L^\infty}$ and $\|\Delta u\|_{L^\infty(0,T;L^d)}$.
This combined with Remark \ref{rmk_ip} completes the proof.
 \end{proof}

%
%
%
%
%

\subsection{General case: $0\vee (d-2j-2) <\alpha < d-2j$, $j \in \N$}
From now on, we consider the general case. We first observe that for $0 \vee (d-2j-2) < \alpha < d-2j$ with $j \in \N$
 \[
 \xi^\gamma (-\Delta_\gamma)^{j+1} K\star(\rho\otimes\delta_0)(x,\xi) = c_{\alpha,d,j} \rho(x)\otimes \delta(\xi) \quad \mbox{in} \quad \R^d \times \R_+ 
\]
for some $c_{\alpha,d,j} > 0$, which will be again assumed to be unity in the sequel. Here $\gamma = \alpha - d + 2j +1 \in (-1,1)$. Note that
\[
\int_{\R^d} (K\star\rho)\rho\, dx =  \left\{\begin{array}{lcll}\displaystyle \iint_{\R^d \times \R_+}  \xi^\gamma \lt| (-\Delta_\gamma)^{\frac{j+1}{2}} K\star\rho\rt|^2 dxd\xi & \mbox{if $j$ is odd}, & \\[4mm]
\displaystyle \iint_{\R^d \times \R_+}  \xi^\gamma \lt| \nabla_{(x,\xi)} (-\Delta_\gamma)^{\frac{j}{2}} K\star\rho\rt|^2 dxd\xi & \mbox{if $j$ is even}.  &
 \end{array}\right.
\]
Indeed, for odd $j$, by using the integration by parts, we obtain 
\begin{align*}
\int_{\R^d} (K\star\rho)\rho\, dx &= \iint_{\R^d \times \R_+}  \xi^\gamma (-\Delta_\gamma)^{j+1} K\star(\rho\otimes\delta_0)(x,\xi) K\star(\rho\otimes\delta_0)(x,\xi)\,dxd\xi\cr
&= \iint_{\R^d \times \R_+}  \xi^\gamma \lt| (-\Delta_\gamma)^{\frac{j+1}{2}} K\star\rho\rt|^2 dxd\xi,
\end{align*}
 due to Lemma \ref{lem_dg} (i). On the other hand, if $j$ is even, then we use \eqref{dg} and Lemma \ref{lem_dg} (i) to find
 \begin{align*}
&\int_{\R^d} (K\star\rho)\rho\, dx \cr
&\quad = \iint_{\R^d \times \R_+}  \xi^\gamma (-\Delta_\gamma)^{j+1} K\star(\rho\otimes\delta_0)(x,\xi) K\star(\rho\otimes\delta_0)(x,\xi)\,dxd\xi\cr
&\quad =\iint_{\R^d \times \R_+}  \xi^\gamma (-\Delta_\gamma)^{\frac j2+1} K\star(\rho\otimes\delta_0)(x,\xi) (-\Delta_\gamma)^{\frac j2}K\star(\rho\otimes\delta_0)(x,\xi)\,dxd\xi\cr
&\quad = - \iint_{\R^d \times \R_+}  \nabla_{(x,\xi)} \cdot \lt(\xi^\gamma \nabla_{(x,\xi)} (-\Delta_\gamma)^{\frac j2} K\star(\rho\otimes\delta_0)(x,\xi)\rt) (-\Delta_\gamma)^{\frac j2}K\star(\rho\otimes\delta_0)(x,\xi)\,dxd\xi\cr
&\quad = \iint_{\R^d \times \R_+}  \xi^\gamma \lt| \nabla_{(x,\xi)} \lt((-\Delta_\gamma)^{\frac{j}{2}} K\star\rho\rt)\rt|^2 dxd\xi.
\end{align*}
Before presenting the technical lemma, we make simple but important observations. For notational simplicity, we set
\[
\K_\ell := \int_{\R^d} \frac{(\rho-\bar\rho)(y)}{|(x-y,\xi)|^{\alpha+2\ell}}\,dy, \quad 0 \le \ell\le j, \quad \K_0 := \K.
\]
On the one hand, for $0 \le \ell \le j-1$,
\[
\Delta_{(x,\xi)}\K_\ell = \Delta_{(x,\xi)}\int_{\R^d} \frac{(\rho-\bar\rho)(y)}{|(x-y,\xi)|^{\alpha+2\ell}}\,dy = -(\alpha+2\ell)(d-1-\alpha-2\ell) \K_{\ell+1},
\]
and
\bq\label{rel_dif1}
\frac{\gamma}{\xi}\pa_\xi \K_\ell = -(\alpha+2\ell) \gamma  \K_{\ell+1},
\eq
which means
\bq\label{rel_dif1_2}
\Delta_\gamma \K_\ell = -(\alpha+2\ell)(2j-2\ell)\K_{\ell+1},
\eq
due to $\gamma = \alpha - d + 2j +1$. From this, we can deduce that
\bq\label{rel_dif2}
\Delta_\gamma^\ell \K = (-1)^\ell\lt[ \prod_{q=0}^{\ell-1}(\alpha+2q)(2j-2q)\rt]\K_\ell, \quad 1\le \ell\le j.
\eq
On the other hand, observe that for any ${\bf v} = {\bf v}(x) \in \R^d$ and $p \in \N \cup \{0\}$
\begin{align*}
\frac{\gamma\pa_\xi}{\xi} \lt( ({\bf v}, 0)\cdot \nabla_{(x,\xi)}\pa_x^p \K_\ell\rt) &=\frac{\gamma\pa_\xi}{\xi}\lt( {\bf v}\cdot \nabla_{x}\pa_x^p \K_\ell\rt)\\
&=  {\bf v}\cdot \nabla_{x}\pa_x^p \lt(\frac{\gamma\pa_\xi}{\xi}\K_\ell\rt)\\
&= -(\alpha+2\ell)\gamma   {\bf v}\cdot \nabla_{x}\pa_x^p \K_{\ell+1}\\
&=  -(\alpha+2\ell)\gamma ({\bf v},0)\cdot \nabla_{(x,\xi)}\pa_x^p \K_{\ell+1}
\end{align*}
and
\begin{align*}
&\Delta_{(x,\xi)}\lt( ({\bf v}, 0)\cdot \nabla_{(x,\xi)}\pa_x^p\K_\ell\rt) \\
&=(\Delta_x {\bf v},0)\cdot \nabla_{(x,\xi)}\pa_x^p \K_\ell + 2\nabla_{(x,\xi)}({\bf v},0) : \nabla_{(x,\xi)}^2 \pa_x^p\K_\ell + ({\bf v},0)\cdot \nabla_{(x,\xi)}\pa_x^p\Delta_{(x,\xi)}\K_{\ell} \\
&= (\Delta_x {\bf v},0)\cdot \nabla_{(x,\xi)}\pa_x^p \K_\ell +2\sum_{j=1}^d (\pa_j{\bf v}, 0)\cdot \nabla_{(x,\xi)}\pa_j \pa_x^p \K_\ell -(\alpha+2\ell)(d-1-\alpha-2\ell) ({\bf v},0)\cdot \nabla_{(x,\xi)}\pa_x^p\K_{\ell+1}.
\end{align*}
This together with \eqref{rel_dif1_2} yields
\begin{align}\label{rel_gam}
\begin{aligned}
&\Delta_\gamma\lt( ({\bf v}, 0)\cdot \nabla_{(x,\xi)}\pa_x^p\K_\ell\rt) \\
&= (\Delta_x {\bf v},0)\cdot \nabla_{(x,\xi)}\pa_x^p \K_\ell +2\sum_{j=1}^d (\pa_j {\bf v}, 0)\cdot \nabla_{(x,\xi)}\pa_j \pa_x^p \K_\ell + ({\bf v},0)\cdot \nabla_{(x,\xi)}\pa_x^p\Delta_\gamma\K_{\ell}\\
&= \Delta_x {\bf v}\cdot\nabla_x \pa_x^p \K_\ell +2\sum_{j=1}^d \pa_j {\bf v} \cdot \nabla_x \pa_j \pa_x^p \K_\ell + {\bf v}\cdot \nabla_x\pa_x^p \Delta_\gamma\K_{\ell}.
\end{aligned}
\end{align}
Thus, we obtain
\bq\label{rel_gam_0}\begin{aligned}
&\Delta_\gamma((u,0)\cdot \nabla_{(x,\xi)}\K)\\
&= \Delta_\gamma((u,0)\cdot \nabla_{(x,\xi)}\K_0)\\
&=  (\Delta_x u,0)\cdot \nabla_{(x,\xi)}\K_0 +2\sum_{j=1}^d (\pa_j u, 0)\cdot \nabla_{(x,\xi)}\pa_j \K_0 +(u,0)\cdot \nabla_{(x,\xi)}\Delta_\gamma \K_0\\
&=  \Delta_x u\cdot\nabla_x  \K_0 +2\sum_{j=1}^d \pa_j u \cdot \nabla_x \pa_j \K_0 -\alpha(4m-2) u\cdot \nabla_x\K_1,
\end{aligned}\eq
and moreover,
\[\begin{aligned}
&\Delta_\gamma^2((u,0)\cdot \nabla_{(x,\xi)}\K)\\
&= \Delta_\gamma \lt((\Delta_x u,0)\cdot \nabla_{(x,\xi)}\K\rt) + \Delta_\gamma\lt(2\sum_{j=1}^d (\pa_j u, 0) \cdot  \nabla_{(x,\xi)}\pa_j\K \rt) + \Delta_\gamma \lt( (u,0) \cdot \nabla_{(x,\xi)} \Delta_\gamma \K\rt) \\
&= (\Delta_x^2 u,0)\cdot \nabla_{(x,\xi)}\K_0 + 4\sum_{j=1}^d (\Delta_x \pa_j u,0)\cdot \nabla_{(x,\xi)}\pa_j \K_0 + 2(\Delta_x u,0)\cdot \nabla_{(x,\xi)}\Delta_\gamma\K_0\\
&\quad + 4\sum_{j_1=1}^d\sum_{j_2=1}^d (\pa_{j_1} \pa_{j_2}u,0)\cdot\nabla_{(x,\xi)}\pa_{j_1} \pa_{j_2} \K_0 +4\sum_{j=1}^d(\pa_j u,0) \cdot \nabla_{(x,\xi)} \pa_j \Delta_\gamma\K_0 +(u,0)\cdot\nabla_{(x,\xi)}\Delta_\gamma^2 \K_0.\\
\end{aligned}\]
Hence, we can easily deduce from \eqref{rel_gam} and \eqref{rel_gam_0} that $\Delta_\gamma^m((u,0)\cdot\nabla_{(x,\xi)}\K)$ consists of
\[
(\Delta_{x}^{\ell_1} \pa_x^\beta u,0)\cdot \nabla_{(x,\xi)} \pa_x^\beta \Delta_\gamma^{\ell_2}\K = \Delta_{x}^{\ell_1} \pa_x^\beta u \cdot \nabla_{x} \pa_x^\beta  \Delta_\gamma^{\ell_2} \K , \quad \ell_1 + \ell_2 + |\beta| = m,
\]
which can be proved by induction on $m$.

\begin{lemma}\label{lem_leib}
Suppose $d-2j-2<\alpha<d-2j$ for some $j \in \N$ and let $\gamma := \alpha-d+2j+1 \in (-1,1)$. Then we have
\[
\Delta_\gamma^m ((u,0)\cdot\nabla_{(x,\xi)}\K) = \sum_{\ell_1+\ell_2+p = m}\frac{m!}{\ell_1!\ell_2!p!} 2^p\sum_{j_1,\dots,j_p=1}^d (\Delta_x^{\ell_1} \pa_{j_1} \cdots \pa_{j_p}u,0) \cdot \nabla_{(x,\xi)} (\pa_{j_1}\cdots\pa_{j_p} \Delta_\gamma^{\ell_2}\K)
\]
for every $0\le m \le j$.
\end{lemma}
\begin{proof} Since the proof is rather lengthy and technical, for the smooth flow of reading, we postpone it to Appendix \ref{app:lem_leib}.
\end{proof}

Then we move on to the next claim.
\begin{lemma}\label{lem_tech2} For $p, \ell \in\N$ with $p+2\ell \le j+1$ and $p \ge 2$, the following inequality holds:
\begin{align}\label{even_2}
\begin{aligned}
&\iint_{\R^d\times\R_+} \xi^\gamma |\nabla_x^{p} \Delta_\gamma^\ell \K|^2\,dxd\xi +\iint_{\R^d\times\R_+} \xi^\gamma |\pa_\xi \nabla_x^{p-1} \Delta_\gamma^\ell \K|^2\,dxd\xi\cr
&\quad \le \lt[\prod_{q=0}^{p-1} \lt(1-\frac{\gamma}{(2j-2\ell-2q)^2}\rt)\rt] \iint_{\R^d\times\R_+} \xi^\gamma |\nabla_x^{p-2[p/2]}\Delta_\gamma^{[p/2]+\ell} \K|^2\,dxd\xi,
\end{aligned}
\end{align}
where each term in the product is positive.
\end{lemma}
For the proof, we give a technical lemma below.
\begin{lemma}\label{lem_tech} For any $\gamma \in (-1,1)$, we have
\[
\iint_{\R^d \times \R_+} \xi^\gamma \pa_\xi \Delta_\gamma f \, \pa_\xi f\,dxd\xi = -\gamma \iint_{\R^d \times \R_+} \xi^{\gamma-2} |\pa_\xi f|^2 \,dxd\xi - \iint_{\R^d \times \R_+} \xi^\gamma |\pa_\xi \nabla_{(x,\xi)} f|^2\,dxd\xi.
\]
\end{lemma}
\begin{proof} By definition of $\Delta_\gamma$, we easily find
\begin{align*}
&\iint_{\R^d \times \R_+} \xi^\gamma \pa_\xi \Delta_\gamma f \, \pa_\xi f\,dxd\xi \cr
&\quad = \iint_{\R^d \times \R_+} \xi^\gamma \pa_\xi \Delta_{(x,\xi)} f \, \pa_\xi f\,dxd\xi + \iint_{\R^d \times \R_+} \xi^\gamma \pa_\xi (\gamma \xi^{-1} \pa_\xi f) \, \pa_\xi f\,dxd\xi\cr
&\quad = - \iint_{\R^d \times \R_+} \xi^\gamma |\pa_\xi \nabla_{(x,\xi)} f|^2 \, dxd\xi - \iint_{\R^d \times \R_+} \pa_\xi (\xi^\gamma) \pa_\xi^2 f \, \pa_\xi f\,dxd\xi\cr
&\qquad - \gamma \iint_{\R^d \times \R_+} \xi^{\gamma-2} |\pa_\xi f|^2 \,dxd\xi + \gamma \iint_{\R^d \times \R_+} \xi^{\gamma-1} \pa_\xi^2 f \, \pa_\xi f \,dxd\xi. 
\end{align*}
\end{proof}
\begin{proof}[Proof of Lemma \ref{lem_tech2}] We first claim that for $p \in \N$
\begin{align}\label{claim_est}
\begin{aligned}
&\iint_{\R^d\times\R_+} \xi^\gamma |\nabla_x^{p} \Delta_\gamma^\ell \K|^2\,dxd\xi +\iint_{\R^d\times\R_+} \xi^\gamma |\pa_\xi \nabla_x^{p-1} \Delta_\gamma^\ell \K|^2\,dxd\xi \cr
&\quad = \lt(1-\frac{\gamma}{(2j-2\ell)^2}\rt) \iint_{\R^d\times\R_+} \xi^\gamma |\nabla_x^{p-2} \Delta_\gamma^{\ell+1} \K|^2\,dxd\xi   - \iint_{\R^d\times\R_+} \xi^\gamma |\pa_\xi \nabla_{(x,\xi)}\nabla_x^{p-2} \Delta_\gamma^\ell \K|^2\,dxd\xi.
\end{aligned}
\end{align}
Note that 
\begin{align*}
&\iint_{\R^d\times\R_+} \xi^\gamma \nabla_x^2 \pa_x^{p-2} \Delta_\gamma^\ell \K : \nabla_x^2 \pa_x^{p-2} \Delta_\gamma^\ell \K\,dxd\xi \cr
&\quad = \iint_{\R^d\times\R_+} \xi^\gamma \nabla_{(x,\xi)}\nabla_x \pa_x^{p-2} \Delta_\gamma^\ell \K : \nabla_{(x,\xi)}\nabla_x  \pa_x^{p-2} \Delta_\gamma^\ell \K\,dxd\xi - \iint_{\R^d\times\R_+} \xi^\gamma |\pa_\xi \nabla_x \pa_x^{p-2}\Delta_\gamma^\ell \K|^2\,dxd\xi.
\end{align*}
On the other hand, the first term on the right hand side of the above equality can be estimated as
\begin{align*}
 & - \iint_{\R^d\times\R_+} \xi^\gamma \Delta_\gamma \nabla_x \pa_x^{p-2}\Delta_\gamma^\ell  \K \cdot \nabla_x \pa_x^{p-2}\Delta_\gamma^\ell \K\,dxd\xi\cr
 &\quad = - \iint_{\R^d\times\R_+} \xi^\gamma \nabla_{(x,\xi)} \pa_x^{p-2} \Delta_\gamma^{\ell+1} \K \cdot \nabla_{(x,\xi)} \pa_x^{p-2} \Delta_\gamma^\ell \K\,dxd\xi \cr
 &\qquad+ \iint_{\R^d\times\R_+} \xi^\gamma \pa_\xi  \pa_x^{p-2}\Delta_\gamma^{\ell+1} \K \, \pa_\xi \pa_x^{p-2} \Delta_\gamma^\ell \K\,dxd\xi\cr
 &\quad = \iint_{\R^d\times\R_+} \xi^\gamma |\pa_x^{p-2} \Delta_\gamma^{\ell+1} \K|^2\,dxd\xi - \gamma \iint_{\R^d\times\R_+} \xi^\gamma \lt|  \pa_x^{p-2}\frac{\pa_\xi}{\xi}\lt( \Delta_\gamma^\ell \K\rt)\rt|^2\,dxd\xi\cr 
 &\qquad - \iint_{\R^d\times\R_+} \xi^\gamma |\pa_\xi \nabla_{(x,\xi)}\pa_x^{p-2} \Delta_\gamma^\ell\K|^2\,dxd\xi,
\end{align*}
due to Lemma \ref{lem_tech}. Here, we combine \eqref{rel_dif1} with \eqref{rel_dif2} to get
\[
\frac{\pa_\xi}{\xi}\lt( \Delta_\gamma^\ell \K\rt) = (-1)^\ell \lt[\prod_{q=0}^{\ell-1}(\alpha+2q)(2j-2q) \rt]\frac{\pa_\xi}{\xi} \K_\ell = \frac{1}{2j-2\ell}\Delta_\gamma^{\ell+1}\K,
\]
and this gives the desired relation. Here, note that $p \ge 2$ implies $2\ell \le j-1$, which implies $|\gamma/(2j-2\ell)^2| \le 1/4.$\\

We now prove the first assertion by induction on $p$. If $p=2$ or 3, it simply follows from \eqref{claim_est}. Let us assume that \eqref{even_2} holds for some $p \geq 3$.\\

First, notice from \eqref{claim_est} that 
\begin{align*}
&\iint_{\R^d\times\R_+} \xi^\gamma |\nabla_x^{p + 1} \Delta_\gamma^\ell\K|^2\,dxd\xi +\iint_{\R^d\times\R_+} \xi^\gamma |\pa_\xi \nabla_x^{p} \Delta_\gamma^\ell \K|^2\,dxd\xi\cr
&\quad =\lt(1-\frac{\gamma}{(2j-2\ell)^2}\rt) \iint_{\R^d\times\R_+} \xi^\gamma |\nabla_x^{p-1} \Delta_\gamma^{\ell+1}  \K|^2\,dxd\xi - \iint_{\R^d\times\R_+} \xi^\gamma |\pa_\xi \nabla_{(x,\xi)}\nabla_x^{p-1}  \Delta_\gamma^\ell \K|^2\,dxd\xi\\
&\le \lt(1-\frac{\gamma}{(2j-2\ell)^2}\rt) \iint_{\R^d\times\R_+} \xi^\gamma |\nabla_x^{p-1} \Delta_\gamma^{\ell+1}  \K|^2\,dxd\xi .
\end{align*}
On the other hand, by the assumption of the inductive hypothesis, we get
\begin{align*}
&\iint_{\R^d\times\R_+} \xi^\gamma |\nabla_x^{p-1} \Delta_\gamma^{\ell+1}  \K|^2\,dxd\xi \cr
&\le \lt[\prod_{q=0}^{[(p-1)/2]-1} \lt(1-\frac{\gamma}{(2j-2\ell-2q)^2} \rt)\rt] \iint_{\R^d\times\R_+} \xi^\gamma |\nabla_x^{p-1-2[(p-1)/2]}\Delta_\gamma^{\ell+[(p-1)/2]+1} \K|^2\,dxd\xi. \\
\end{align*}
Since $|\gamma/(2j-2\ell)^2|\le 1/4$, combining the above two inequalities concludes the assertion.
\end{proof}

As a direct consequence of Lemma \ref{lem_tech2}, we have the following proposition.
\begin{proposition}
For $0 \vee (d-2j - 2) < \alpha < d- 2j$ with $j \in \N$, we have
\[
\iint_{\R^d\times\R_+} \xi^\gamma |\nabla_x^{j+1} \K|^2\,dxd\xi  \leq  C\left\{\begin{array}{lcll}\displaystyle \iint_{\R^d \times \R_+}  \xi^\gamma | (-\Delta_\gamma)^{\frac{j+1}{2}} \K |^2 dxd\xi & \mbox{if $j$ is odd} & \\[4mm]
\displaystyle \iint_{\R^d \times \R_+}  \xi^\gamma | \nabla_x (-\Delta_\gamma)^{\frac{j}{2}} \K  |^2 dxd\xi & \mbox{if $j$ is even}  &
 \end{array}\right.,
\]
where $\gamma = \alpha - d + 2m + 1 \in (0,1)$ and $C=C(\alpha, d, j)$ is a positive constant.
\end{proposition}

Now we provide the details on the proof of Theorem \ref{thm_main} in the general case.

 \begin{proof}[Proof of Theorem \ref{thm_main}] 
 We separately estimate two cases as follows:\\
 
 \noindent $\diamond$ (Case A: $d-4m<\alpha<d-4m+2$, $m \in \N$) Note that this corresponds to the case when $j=2m-1$ is odd. We use Lemma \ref{lem_leib} and the following type of Gagliardo--Nirenberg interpolation inequality
 \bq\label{gn_ineq}
 \|f\|_{L^{\frac dk}} \le C \|\nabla f\|_{L^{\frac{d}{k+1}}}, \quad k+1<d
 \eq
 to estimate
  \begin{align*}
&\int_{\R^d} (\rho-\bar\rho) u \cdot \nabla K \star(\rho-\bar\rho)\,dx\\
&=\iint_{\R^d\times\R_+} \xi^\gamma \Delta_\gamma^m \K \Delta_\gamma^m( (u,0)\cdot \nabla_{(x,\xi)}\K)\,dxd\xi\\
&= -\frac12\iint_{\R^d\times\R_+}(\nabla_x \cdot u) \xi^\gamma |\Delta_\gamma^m \K|^2\,dxd\xi + \iint_{\R^d\times\R_+}\xi^\gamma \Delta_\gamma^m \K (\Delta_x^m u,0)\cdot \nabla_{(x,\xi)}\K\,dxd\xi\\
&\quad + \iint_{\R^d\times\R_+} \xi^\gamma \Delta_\gamma^m \K \sum_{\substack{\ell_1+\ell_2+p = m\\ \ell_1,\ell_2 \neq m}}\frac{m!}{\ell_1!\ell_2!p!} 2^p\sum_{j_1,\dots,j_p=1}^d (\Delta_x^{\ell_1} \pa_{j_1} \cdots \pa_{j_p}u,0) \cdot \nabla_{(x,\xi)} (\pa_{j_1}\cdots\pa_{j_p} \Delta_\gamma^{\ell_2}\K)\,dxd\xi\\
&= -\frac12\iint_{\R^d\times\R_+}(\nabla_x \cdot u) \xi^\gamma |\Delta_\gamma^m \K|^2\,dxd\xi + \iint_{\R^d\times\R_+}\xi^\gamma \Delta_\gamma^m \K \lt(\Delta_x^m u\cdot \nabla_x \K\rt)\,dxd\xi\\
&\quad + \iint_{\R^d\times\R_+} \xi^\gamma \Delta_\gamma^m \K \sum_{\substack{\ell_1+\ell_2+p = m\\ \ell_1,\ell_2 \neq m}}\frac{m!}{\ell_1!\ell_2!p!} 2^p\sum_{j_1,\dots,j_p=1}^d \Delta_x^{\ell_1} \pa_{j_1} \cdots \pa_{j_p}u \cdot \nabla_x (\pa_{j_1}\cdots\pa_{j_p} \Delta_\gamma^{\ell_2}\K)\,dxd\xi\\
&\le \frac{\|\nabla\cdot u\|_{L^\infty}}{2}\iint_{\R^d\times\R_+} \xi^\gamma |\Delta_\gamma^m \K|^2\,dxd\xi +\int_{\R_+} \xi^\gamma \|\Delta_\gamma^m \K\|_{L_x^2} \|\Delta_x^m u\cdot \nabla_x \K\|_{L_x^2}\,d\xi\\
&\quad +C \sum_{\substack{\ell_1+\ell_2+p = m\\ \ell_1,\ell_2 \neq m}}\sum_{j_1,\dots,j_p=1}^d \int_{\R_+}\xi^\gamma \|\Delta_\gamma^m \K\|_{L_x^2} \|\Delta_x^{\ell_1} \pa_{j_1} \cdots \pa_{j_p}u   \cdot \nabla_x (\pa_{j_1}\cdots\pa_{j_p} \Delta_\gamma^{\ell_2}\K)\|_{L_x^2}  \,d\xi \\
&\le \frac{\|\nabla\cdot u\|_{L^\infty}}{2}\iint_{\R^d\times\R_+} \xi^\gamma |\Delta_\gamma^m \K|^2\,dxd\xi + \int_{\R_+} \xi^\gamma \|\Delta_\gamma^m \K\|_{L_x^2} \|\Delta_x^m u\|_{L_x^{\frac{d}{2m-1}}}\|\nabla_x \K\|_{L_x^{\frac{2d}{d-2(2m-1)}}}\\
&\quad +C  \sum_{\substack{\ell_1+\ell_2+p = m\\ \ell_1,\ell_2 \neq m}} \int_{\R_+}\xi^\gamma \|\Delta_\gamma^m \K\|_{L_x^2} \|\nabla_x^{2\ell_1+p} u\|_{L^{\frac{d}{2\ell_1+p-1}}}  \|\nabla_x^{p+1}  \Delta_\gamma^{\ell_2}\K\|_{L_x^{\frac{2d}{d-2(2\ell_1+p-1)}}}  \,d\xi\\
&\le C\iint_{\R^d\times\R_+} \xi^\gamma |\Delta_\gamma^m \K|^2\,dxd\xi +\int_{\R_+} \xi^\gamma \|\Delta_\gamma^m \K\|_{L_x^2} \|\Delta_x^m u\|_{L_x^{\frac{d}{2m-1}}}\|\nabla_x^{2m} \K\|_{L_x^2} \\
&\quad +C  \sum_{\substack{\ell_1+\ell_2+p = m\\ \ell_1,\ell_2 \neq m}} \int_{\R_+}\xi^\gamma \|\Delta_\gamma^m \K\|_{L_x^2} \|\nabla_x^{2m-1}u\|_{L^{\frac{d}{2m-2}}}  \| \nabla_x^{2\ell_1 + 2p}  \Delta_\gamma^{\ell_2}\K\|_{L_x^2}  \,d\xi\\
&\le C\iint_{\R^d\times\R_+} \xi^\gamma |\Delta_\gamma^m \K|^2\,dxd\xi +C \|\xi^{\gamma/2} \Delta_\gamma^m \K\|_{L_{x,\xi}^2} \sum_{\ell_2=0}^{m-1}\lt(\iint_{\R^d\times\R_+}\xi^\gamma  |\nabla_x^{2(m-\ell_2)} \Delta_\gamma^{\ell_2}\K|^2  \,d\xi \rt)^{1/2}\\
&\le  C\iint_{\R^d\times\R_+} \xi^\gamma |\Delta_\gamma^m \K|^2\,dxd\xi,
\end{align*}
where $C=C(\|\nabla\cdot u\|_{L^\infty}, \|\Delta_x^m u\|_{L^{\frac{d}{2m-1}}}, \|\nabla_x^{2m-1} u\|_{L^{\frac{d}{2m-2}}}, \alpha, d, m)$ is a positive constant and we again denoted by  $\K = K\star((\rho-\bar\rho)\otimes \delta_0)$. Since $d-4m<\alpha<d-4m+2$, we get
\[
m = \lt[\frac{d-\alpha}{4}\rt], \quad \frac{d}{2m-1} = \frac{d}{[(d-\alpha)/2]}, \quad \mbox{and}\quad \frac{d}{2m-2} = \frac{d}{[(d-\alpha)/2]-1}.
\]
\vspace{0.3cm}

 \noindent $\diamond$ (Case B: $d-4m-2<\alpha<d-4m$, $m \in \N$) Note that this corresponds to the case when $j=2m$ is even. Similarly to the above, we use Lemma \ref{lem_leib} to yield
\begin{align*}
&\int_{\R^d} (\rho-\bar\rho) u \cdot \nabla K \star(\rho-\bar\rho)\,dx\\
&=\iint_{\R^d\times\R_+} \xi^\gamma \Delta_\gamma^{m+1} \K \Delta_\gamma^m( (u,0)\cdot \nabla_{(x,\xi)}\K)\,dxd\xi\\
&= \iint_{\R^d\times\R_+}\xi^\gamma \lt(\Delta_\gamma^{m+1} \K\rt) (u,0)\cdot \nabla_{(x,\xi)}\Delta_\gamma^m \K\,dxd\xi + \iint_{\R^d\times\R_+}\xi^\gamma\lt( \Delta_\gamma^{m+1} \K \rt)(\Delta_x^m u,0)\cdot \nabla_{(x,\xi)}\K\,dxd\xi\\
&\quad + \iint_{\R^d\times\R_+} \xi^\gamma \Delta_\gamma^{m+1} \K \sum_{\substack{\ell_1+\ell_2+p = m\\ \ell_1,\ell_2 \neq m}}\frac{m!}{\ell_1!\ell_2!p!} 2^p\sum_{j_1,\dots,j_p=1}^d (\Delta_x^{\ell_1} \pa_{j_1} \cdots \pa_{j_p}u,0) \cdot \nabla_{(x,\xi)} (\pa_{j_1}\cdots\pa_{j_p} \Delta_\gamma^{\ell_2}\K)\,dxd\xi\\
&=:\sfK_1 + \sfK_2 + \sfK_3.
\end{align*}
For $\sfK_1$,
\[\begin{aligned}
\sfK_1&= -\iint_{\R^d\times\R_+}\xi^\gamma \lt(\nabla_{(x,\xi)}\Delta_\gamma^m\K \otimes \nabla_{(x,\xi)}\Delta_\gamma^m \K \rt): \nabla_{(x,\xi)}(u,0)\,dxd\xi\\
&\quad + \frac12\iint_{\R^d\times\R_+} (\nabla_x\cdot u)\xi^\gamma |\nabla_{(x,\xi)} \Delta_\gamma^m \K|^2\,dxd\xi\\
&\le C\iint_{\R^d\times\R_+} \xi^\gamma |\nabla_{(x,\xi)} \Delta_\gamma^m \K|^2\,dxd\xi,
\end{aligned}\]
where $C=C(\|\nabla u\|_{L^\infty})$ is a positive constant.\\

\noindent For $\sfK_2$, we use Lemma \ref{lem_tech2} to get
\[\begin{aligned}
\sfK_2 &= -\iint_{\R^d\times\R_+} \xi^\gamma \nabla_{x} \Delta_\gamma^m \K \cdot \nabla_x(\Delta_x^m u) \cdot \nabla_x\K\,dxd\xi\\
&\quad -\iint_{\R^d\times\R_+} \xi^\gamma \nabla_{(x,\xi)} \Delta_\gamma^m \K \cdot \lt(\nabla_{(x,\xi)}\nabla_x \K \cdot \Delta_x^m u \rt)\,dxd\xi\\
&\le \int_{\R_+} \xi^\gamma \|\nabla_x \Delta_\gamma^m \K\|_{L_x^2}\|\nabla_x \Delta_x^m u\|_{L_x^{\frac{d}{2m}}} \|\nabla_x \K\|_{L_x^{\frac{2d}{d-4m}}}d\xi\\
&\quad + \int_{\R_+} \xi^\gamma \|\nabla_x \Delta_\gamma^m \K\|_{L_x^2}\| \Delta_x^m u\|_{L_x^{\frac{d}{2m-1}}} \|\nabla_{(x,\xi)}(\nabla_x \K)\|_{L_x^{\frac{2d}{d-2(2m-1)}}}d\xi\\
&\le C\int_{\R_+}\xi^\gamma \|\nabla_x \Delta_\gamma^m \K\|_{L_x^2} \|\nabla_{(x,\xi)}(\nabla_x^{2m}\K)\|_{L_x^2}\,d\xi\\
&\le C\iint_{\R^d\times\R_+} \xi^\gamma |\nabla_{(x,\xi)} \Delta_\gamma^m \K|^2\,dxd\xi,
\end{aligned}\]
where $C=C(\|\nabla_x \Delta_x^m u\|_{L^{\frac{d}{2m}}}, \|\nabla_x^{2m} u\|_{L^{\frac{d}{2m-1}}}, \alpha,d,m)$ is a positive constant.\\

\noindent Finally, for $\sfK_3$, we also use Lemma \ref{lem_tech2} and \eqref{gn_ineq} to get
\begin{align*}
\sfK_3&= -\sum_{\substack{\ell_1+\ell_2+p = m\\ \ell_1,\ell_2 \neq m\\ 1\le j_1, \dots, j_p \le d}} \frac{m!}{\ell_1!\ell_2!p!} 2^p \iint_{\R^d\times\R_+} \xi^\gamma \nabla_{(x,\xi)}\Delta_\gamma^m \K  \cr
&\hspace{6cm}\cdot \nabla_{(x,\xi)} \lt[(\Delta_x^{\ell_1} \pa_{j_1} \cdots \pa_{j_p}u,0) \cdot \nabla_{(x,\xi)} (\pa_{j_1}\cdots\pa_{j_p} \Delta_\gamma^{\ell_2}\K)\rt]\,dxd\xi\\
&= - \sum_{\substack{\ell_1+\ell_2+p = m\\ \ell_1,\ell_2 \neq m \\ 1\le j_1,\dots,j_p\le d}} \frac{m!}{\ell_1!\ell_2!p!} 2^p\iint_{\R^d\times\R_+} \xi^\gamma \lt(\nabla_{(x,\xi)}\Delta_\gamma^m \K \otimes   \nabla_{(x,\xi)} (\pa_{j_1}\cdots\pa_{j_p} \Delta_\gamma^{\ell_2}\K)\rt) \cr
&\hspace{9cm}: \nabla_{(x,\xi)}(\Delta_x^{\ell_1} \pa_{j_1} \cdots \pa_{j_p}u,0) \,dxd\xi\\
&\quad -  \sum_{\substack{\ell_1+\ell_2+p = m\\ \ell_1,\ell_2 \neq m \\ 1\le j_1,\dots,j_p\le d}} \frac{m!}{\ell_1!\ell_2!p!} 2^p\iint_{\R^d\times\R_+} \xi^\gamma \nabla_{(x,\xi)}\Delta_\gamma^m \K \cdot   \nabla_{(x,\xi)}^2 (\pa_{j_1}\cdots\pa_{j_p} \Delta_\gamma^{\ell_2}\K)\cdot (\Delta_x^{\ell_1} \pa_{j_1} \cdots \pa_{j_p}u,0) \,dxd\xi\\
&\le C \sum_{\substack{\ell_1+\ell_2+p = m\\ \ell_1,\ell_2 \neq m}}\int_{\R_+} \xi^\gamma \|\nabla_{(x,\xi)}\Delta_\gamma^m \K\|_{L_x^2} \|\nabla_{(x,\xi)}\nabla_x^p \Delta_\gamma^{\ell_2}\K\|_{L^{\frac{2d}{d-2(2\ell_1 + p)}}} \|\nabla_x^{2\ell_1 + p+1} u\|_{L_x^{\frac{d}{2\ell_1 + p}}}\,d\xi\\
&\quad +C \sum_{\substack{\ell_1+\ell_2+p = m\\ \ell_1,\ell_2 \neq m}}\int_{\R_+} \xi^\gamma \|\nabla_{(x,\xi)}\Delta_\gamma^m \K\|_{L_x^2} \|\nabla_{(x,\xi)}\nabla_x^{p+1} \Delta_\gamma^{\ell_2}\K\|_{L^{\frac{2d}{d-2(2\ell_1 + p-1)}}} \|\nabla_x^{2\ell_1 + p} u\|_{L_x^{\frac{d}{2\ell_1 + p-1}}}\,d\xi\\
&\le C\sum_{\substack{\ell_1+\ell_2+p = m\\ \ell_1,\ell_2 \neq m}}\int_{\R_+}  \xi^\gamma \|\nabla_{(x,\xi)}\Delta_\gamma^m \K\|_{L_x^2} \|\nabla_x^{2m} u\|_{L^{\frac{d}{2m-1}}} \|\nabla_{(x,\xi)}\nabla_x^{2(m-\ell_2)}\Delta_\gamma^{\ell_2} \K\|_{L_x^2}\,d\xi\\
&\le C\iint_{\R^d\times\R_+} \xi^\gamma |\nabla_{(x,\xi)} \Delta_\gamma^m \K|^2\,dxd\xi,
\end{align*}
where $C=C(\|\nabla_x^{2m} u\|_{L^{\frac{d}{2m-1}}}, \alpha,d,m)$ is a positive constant. Since $d-4m-2<\alpha<d-4m$, we find
\[
m = \lt[\frac{d-\alpha}{4}\rt], \quad \frac{d}{2m} = \frac{d}{[(d-\alpha)/2]},\quad  \mbox{and} \quad \frac{d}{2m-1} = \frac{d}{[(d-\alpha)/2]-1}.
\]
This concludes the desired results.
\end{proof}

%
%
%
%
%
%
\section{Modulated interaction energy estimates: $\alpha \equiv d$ mod 2.}\label{sec_subCoul2}

In this section, we cover the remaining case, where the exponent $\alpha$ is given as a positive integer equivalent to $d$ modulo 2, i.e. $K$ has the form of
\bq\label{kernel_sing}
K(x) = |x|^{-(d-2m)} \quad \mbox{where} \quad 2m \in \bbn, \quad 2m < d-2.
\eq
In this case, we do not need to introduce the elliptic operator $\Delta_\gamma$ appeared in \eqref{def_dg} to have the modulated energy estimate in Theorem \ref{thm_main}. More precisely, let us take into account two cases; (i) $m$ is even and (ii) $m$ is odd. Then, as mentioned before, in case (i), we can interpret the potential $K$ satisfies 
\[
(-\Delta)^{m} K(x) = c_{m,d}\delta_0(x)
\]
for some $c_{m,d} > 0$. For simplicity of the presentation, we again set $c_{m,d} = 1$ in the rest of this section. From those observations, we have the following identities for the interaction energy:
\[
\int_{\R^d} (K\star\rho)\rho\, dx =  \left\{\begin{array}{lcll}\displaystyle \int_{\R^d} \lt| (-\Delta)^{\frac{m}{2}} K\star\rho\rt|^2\,dx & \mbox{if} & m\equiv 0 &\mbox{mod 2},\\
\ & \ & \ & \\
\displaystyle \int_{\R^d} \lt| \nabla (-\Delta)^{\frac{m-1}{2}} K\star\rho\rt|^2\,dx & \mbox{if} & m \equiv 1 & \mbox{mod 2}.
 \end{array}\right. 
\]

Before we proceed, we give the following modified Moser-type inequality.
\begin{lemma}\label{Moser_lem}
For given $d, k \in \bbn$ with $d\ge2k$. Suppose that $f\in H^k(\R^d)$, $g \in W^{k,\frac{d}{k-1}}(\R^d)$ and $\nabla g \in L^\infty(\R^d)$. Then we have
\[
\|\nabla^k (fg) - f\nabla^k g\|_{L^2} \le C \|\nabla^{k-1} f\|_{L^2}\lt(\|\nabla g\|_{L^\infty} + \|\nabla^k g\|_{L^{\frac{d}{k-1}}} \rt),
\]
where $C> 0$ depends on $k$ and $d$.
\end{lemma}
\begin{proof}
We use \eqref{gn_ineq} to get
\[\begin{aligned}
\|\nabla^k (fg) - f\nabla^k g\|_{L^2}&= \left\| \sum_{\ell=1}^{k}\binom{k}{\ell} (\nabla^{k-\ell} f) (\nabla^\ell g)\right\|_{L^2}\\
&\le C \|\nabla^{k-1} f\|_{L^2} \|\nabla g\|_{L^\infty} + C\sum_{\ell=2}^k \|(\nabla^{k-\ell} f) (\nabla^\ell g)\|_{L^2}\\
&\le C \|\nabla^{k-1} f\|_{L^2} \|\nabla g\|_{L^\infty} + C\sum_{\ell=2}^k \|\nabla^{k-\ell} f\|_{L^{\frac{2d}{d-2(\ell-1)}}} \|\nabla^\ell g\|_{L^{\frac{d}{\ell-1}}}\\
&\le C\|\nabla^{k-1} f\|_{L^2}\lt( \|\nabla g\|_{L^\infty} +  \|\nabla^k g\|_{L^{\frac{d}{k-1}}}\rt),
\end{aligned}\]
giving the desired result.
\end{proof}

\begin{proof}[Proof of Theorem \ref{thm_main} when $K$ is given by \eqref{kernel_sing}] We divide the proof into two cases; (Case A) $m$ is even and (Case B) $m$ is odd. \newline

\noindent $\diamond$ (Case A: $m$ is even) We apply $k=m= \frac{d-\alpha}{2}$ to Lemma \ref{Moser_lem} and obtain
\begin{align*} 
 &\int_{\R^d} (\rho-\bar\rho) u \cdot \nabla K \star(\rho-\bar\rho)\,dx\\
&\quad =  \int_{\R^d} (-\Delta)^m (K\star(\rho-\bar\rho)) u \cdot \nabla K \star(\rho-\bar\rho)\,dx\\
&\quad =  \int_{\R^d} (-\Delta)^{\frac m2} K\star(\rho-\bar\rho) (-\Delta)^{\frac m2} (u \cdot \nabla K\star(\rho-\bar\rho))\,dx\\
&\quad = \int_{\R^d} (-\Delta)^{\frac m2} K\star(\rho-\bar\rho)\lt[ (-\Delta)^{\frac m2},  u \cdot \nabla\rt] K\star(\rho-\bar\rho))\,dx\\
&\qquad +  \int_{\R^d}  (-\Delta)^{\frac m2} K\star(\rho-\bar\rho)  u \cdot \nabla((-\Delta)^{\frac m2}K\star(\rho-\bar\rho))\,dx\\
&\quad \le C\|(-\Delta)^{\frac m2} K\star(\rho-\bar\rho)\|_{L^2} \|\nabla^m K\star(\rho-\bar\rho)\|_{L^2} \lt(\|\nabla u\|_{L^\infty} + \|\nabla^m u\|_{L^{\frac{d}{m-1}}}\rt)\\
&\qquad + \frac{\|\nabla \cdot u\|_{L^\infty}}{2}\|(-\Delta)^{\frac m2} K\star(\rho-\bar\rho)\|_{L^2}^2\\
&\quad \le C\|(-\Delta)^{\frac m2} K\star(\rho-\bar\rho)\|_{L^2}^2\cr
&\quad = C\int_{\R^d} (\rho-\bar\rho)K\star(\rho-\bar\rho)\,dx,
\end{align*}
where $C > 0$ depends on $\|\nabla u\|_{L^\infty}$ and $\|\nabla^{\frac{d-\alpha}{2}}u\|_{L^{\frac{2d}{d-\alpha-2}}}$  and we used
\[
\|\Delta w\|_{L^2} = \|\nabla^2 w\|_{L^2}.
\]

\medskip

\noindent $\diamond$ (Case B: $m$ is odd) Similarly to Case A, we can get
\begin{align*} 
 &\int_{\R^d} (\rho-\bar\rho) u \cdot \nabla K \star(\rho-\bar\rho)\,dx\\
 &\quad= \int_{\R^d} \nabla\lt((-\Delta)^{\frac{m-1}{2}} K\star(\rho-\bar\rho)\rt)\cdot \nabla \lt((-\Delta)^{\frac{m-1}{2}} (u \cdot \nabla K\star(\rho-\bar\rho))\rt)\,dx\\
 &\quad=  \int_{\R^d} \nabla\lt((-\Delta)^{\frac{m-1}{2}} K\star(\rho-\bar\rho)\rt)\cdot \nabla^2 \lt((-\Delta)^{\frac{m-1}{2}} K\star(\rho-\bar\rho)\rt) \cdot u\,dx\\
  & \qquad+  \int_{\R^d}\nabla\lt((-\Delta)^{\frac{m-1}{2}} K\star(\rho-\bar\rho)\rt)\cdot \lt[\nabla (-\Delta)^{\frac{m-1}{2}},  u\rt] \cdot \nabla K\star(\rho-\bar\rho) \,dx\\
 &\quad\le \frac{\|\nabla \cdot u\|_{L^\infty}}{2} \|\nabla\lt((-\Delta)^{\frac{m-1}{2}} K\star(\rho-\bar\rho)\rt)\|_{L^2}^2\\
 &\qquad + C\|\nabla\lt((-\Delta)^{\frac{m-1}{2}} K\star(\rho-\bar\rho)\rt)\|_{L^2} \|\nabla^m K\star(\rho-\bar\rho)\|_{L^2}  \lt(\|\nabla u\|_{L^\infty} + \|\nabla^m u\|_{L^{\frac{d}{m-1}}}\rt)\\
 &\quad\le C\|\nabla\lt((-\Delta)^{\frac{m-1}{2}} K\star(\rho-\bar\rho)\rt)\|_{L^2}^2 \cr
 &\quad= C\int_{\R^d} (\rho-\bar\rho)K\star(\rho-\bar\rho)\,dx.
\end{align*}
This completes the proof.
\end{proof}

%
%
%
%
%
%

\section{Applications: Quantified asymptotic analysis}\label{sec:appl}

In the section, we establish quantified asymptotic analysis for kinetic equations to derive the aggregation equation, pressureless/isothermal Euler equations with nonlocal singular interactions forces.

%
%
%
%
%
%

\subsection{Small inertia limit of the kinetic equation}

In this subsection, we provide the details of the proof for Theorem \ref{thm_ktoc} on the small inertia limit of the kinetic equation. Let us recall our main asymptotic problem:
\bq\label{main_kin1}
\e\pa_t f^\e + \e v \cdot \nabla_x f^\e - \nabla_v \cdot \lt((\gamma v   + \nabla K \star \rho^\e )f^\e\rt) =0, \quad (x,v) \in \R^d \times \R^d,  \quad t > 0.
\eq
Without loss of generality, we take $\gamma = 1$ in the rest of this section. 

For the proof of Theorem \ref{thm_ktoc}, we first show the quantitative bound on the modulated energies.

\begin{proposition}\label{prop_ktoc} Let the assumptions of Theorem \ref{thm_main} verified. Then we have
\begin{align*}
&\intr (\rho - \rho^\e)K\star(\rho - \rho^\e)\,dx+ \int_0^t \intrr |v - u(x,s) |^2f^\e(x,v,s)\, dxdv ds \cr
&\quad \leq C\e\intrr |v - u_0(x) |^2 f^\e_0(x,v)\,dxdv + C\intr (\rho_0 - \rho^\e_0)K\star(\rho_0 - \rho^\e_0)\,dx + C\e^2
\end{align*}
for all $t \in [0,T]$ and $\e>0$ small enough, where $C>0$ is independent of $\e$.
\end{proposition}

\begin{remark}\label{rmk_dbl} It follows from \cite{CC20} $($see also \cite{CC21,C21, FK19}$)$ that 
\[
\d_{BL}^2(\rho(\cdot,t), \rho^\e(\cdot,t)) \leq C\d_{BL}^2(\rho_0, \rho^\e_0) + C\int_0^t \intrr |v - u(x,s)|^2 f^\e(x,v,s)\,dxdvds,
\]
where $C>0$ depends on $\|\nabla_x u\|_{L^\infty}$, but independent of $\e>0$. This together with Proposition \ref{prop_ktoc} implies 
\begin{align*}
&\d_{BL}(\rho, \rho^\e) + \intr (\rho - \rho^\e)K\star(\rho - \rho^\e)\,dx+ \int_0^t \intrr |v - u(x,s) |^2f^\e(x,v,s)\, dxdv ds \cr
&\quad \leq C\d_{BL}(\rho_0, \rho^\e_0) + C\e\intrr |v - u_0(x) |^2 f^\e_0(x,v)\,dxdv + C\intr (\rho_0 - \rho^\e_0)K\star(\rho_0 - \rho^\e_0)\,dx + C\e^2,
\end{align*}
where $C>0$ is independent of $\e>0$.
\end{remark}

Let us first begin with the a priori estimate of free energy for the kinetic equation \eqref{main_kin1}. Since its proof is straightforward, we omit it here.
\begin{lemma}\label{lem_free} Let $f^\e$ be a solution to the equation \eqref{main_kin1} with sufficient integrability. Then we have
\[
\frac{d}{dt}\lt(\frac12\intrr |v|^2 f^\e\,dxdv  +  \frac1{2\e}\intr \rho^\e K \star \rho^\e\,dx\rt) + \frac1\e\intrr |v|^2 f^\e\,dxdv = 0.
\]
In particular, this gives
\[
\int_0^t\intrr |v|^2 f^\e\,dxdvds \leq C
\]
for some $C>0$ independent of $\e$.
\end{lemma}

Motivated from \cite{CC21,CPW20,LT13,LT17}, in which the strong relaxation limit from Euler to continuity-type equations are studied, we rewrite the limiting equation \eqref{eq_agg} as
\begin{align}\label{eq_Euler}
\begin{aligned}
&\pa_t \rho + \nabla_x \cdot (\rho u) = 0, \cr
&\e(\pa_t u + u \cdot \nabla_x u) = -  u - \nabla K \star \rho + \e e,
\end{aligned}
\end{align}
where $e := \pa_t u + u \cdot \nabla_x u$.

\begin{proposition}\label{prop_ktoc2} Suppose that the assumptions in Theorem \ref{thm_ktoc} hold. Then we have
\begin{align*}
&\frac12\frac{d}{dt}\lt( \intrr |u-v|^2 f^\e\,dxdv + \frac{1}{\e}\intr (\rho - \rho^\e)K\star(\rho - \rho^\e)\,dx\rt)\cr
&\quad \leq -\lt(\frac{1}{2\e} - \|\nabla_x u\|_{L^\infty}\rt)\intrr |v-u|^2 f^\e\,dxdv  + \frac{C}{\e}\intr (\rho - \rho^\e)K\star(\rho - \rho^\e)\,dx +   \e\|e\|_{L^\infty}^2,
\end{align*}
where $C>0$ is independent of $\e>0$. 
\end{proposition}
\begin{remark} By Proposition \ref{prop_ktoc2}, we also obtain
\begin{align*}
&  \frac1{\e}\intr (\rho - \rho^\e)K\star(\rho - \rho^\e)\,dx+ \intrr |v - u(x,s) |^2f(x,v,t)\, dxdv  \cr
&\quad \leq C(1 + \e)\intrr |v - u_0(x) |^2 f^\e_0(x,v)\,dxdv   + \frac C{\e}\intr (\rho_0 - \rho^\e_0)K\star(\rho_0 - \rho^\e_0)\,dx + C\e.
\end{align*}
In particular if 
\[
\intrr |v - u_0(x) |^2 f^\e_0(x,v)\,dxdv+\frac1\e\intr (\rho_0 - \rho^\e_0)K\star(\rho_0 - \rho^\e_0)\,dx\leq C\e
\]
for some $C>0$ which is independent of $\e$, then we have
\[
\frac1\e\intr (\rho - \rho^\e)K\star(\rho - \rho^\e)\,dx+ \int_0^t \intrr |v - u(x,s) |^2f^\e(x,v,s)\, dxdv ds  \leq C\e
\]
and
\[
\intrr |v - u(x,s) |^2f^\e(x,v,t)\, dxdv \leq C\e^2
\]
for all $t \in [0,T]$, where $C>0$ is independent of $\e$. This shows that the convergence estimates for $\intr v \,f^\e\,dv$ and $f^\e$ in Theorem \ref{thm_ktoc} can be replaced by
\[
\intr v \,f^\e\,dv  \rightharpoonup \rho u  \quad \mbox{weakly in } L^\infty(0,T;\mathcal{M}(\R^d)), 
\]
and
\[
f^\e \rightharpoonup \rho\delta_{u} \quad \mbox{weakly in } L^\infty(0,T;\mathcal{M}(\R^d \times \R^d)),
\]
respectively.
\end{remark}
\begin{proof}[Proof of Proposition \ref{prop_ktoc2}]
A straightforward computation gives
\begin{align*}
\frac12\frac{d}{dt}\intrr |u-v|^2 f^\e\,dxdv &= - \intrr (v-u) \cdot (\pa_t u) f^\e\,dx + \frac12\intrr |v-u|^2 \pa_t f^\e\,dxdv \cr
&=: \sfI_1 + \sfI_2.
\end{align*}
On the other hand, it follows from \eqref{main_kin1} and \eqref{eq_Euler} that 
\begin{align*}
\sfI_2 &= \intrr u\otimes (v-u): (\nabla_x u) f^\e\,dxdv + \frac1\e \intrr (v-u) \cdot \lt(  u   + \nabla K \star \rho\rt)f^\e \,dxdv\cr
&\quad + \intrr (v-u) \cdot e f^\e\,dxdv
\end{align*}
and
\[
\sfI_2 = -\intrr v \otimes (v-u) :(\nabla_x u) f^\e\,dxdv - \frac1\e \intrr (v-u) \cdot (  v   + \nabla K \star \rho^\e)f^\e\,dxdv.
\]
Thus, by combining the above estimates, we obtain
\begin{align*}
&\frac12\frac{d}{dt}\intrr |u-v|^2 f^\e\,dxdv \cr
&\quad = \intrr (u-v) \otimes (v-u): (\nabla_x u) f^\e\,dxdv - \frac1\e \intrr |u-v|^2 f^\e\,dxdv \cr
&\qquad  + \frac1\e \intrr (v-u) \cdot \nabla K  \star (\rho - \rho^\e) f^\e\,dxdv + \intrr (v-u) \cdot e f^\e\,dxdv \cr
&\quad =: \sum_{i=1}^4 \sfJ_1,
\end{align*}
where we easily get
\[
\sfJ_1 \leq \|\nabla_x u\|_{L^\infty}\intrr  |u-v|^2 f^\e\,dxdv 
\]
and
\[
\sfJ_4 \leq \|e\|_{L^\infty}\intrr |v-u| f^\e\,dxdv \leq \e\|e\|_{L^\infty}^2 + \frac1{4\e} \intrr |v-u|^2 f^\e\,dxdv,
\]
due to $\intrr f^\e\,dxdv = 1$ for $t\geq0$.  For the estimate of $\sfJ_3$, by introducing the local moment $\rho^\e u^\e := \intr vf^\e\,dv$, we rewrite
\[
\sfJ_3 =  - \frac1\e \intr \rho^\e (u-u^\e) \cdot \nabla K  \star (\rho - \rho^\e)\,dx.
\]
We then apply Theorem \ref{thm_main} with $(\bar\rho, \bar u) = (\rho^\e, u^\e)$ to deduce
\[
\frac{1}{2\e}\frac{d}{dt}\intr (\rho - \rho^\e)K\star(\rho - \rho^\e)\,dx + \sfJ_4 \leq \frac{C}{\e}\intr (\rho - \rho^\e)K\star(\rho - \rho^\e)\,dx,
\]
where $C>0$ is independent of $\e>0$. 
 
Combining all of the above estimates, we have
\begin{align*}
&\frac12\frac{d}{dt}\lt(\intrr |u-v|^2 f^\e\,dxdv + \frac{1}{\e  }\intr (\rho - \rho^\e)K\star(\rho - \rho^\e)\,dx \rt)\cr
&\quad \leq -\lt(\frac{1}{2\e} - \|\nabla_x u\|_{L^\infty}\rt)\intrr |v-u|^2 f^\e\,dxdv  + \frac{C}{\e}\intr (\rho - \rho^\e)K\star(\rho - \rho^\e)\,dx +  \e\|e\|_{L^\infty}^2.
\end{align*}
Finally, by choosing $\e > 0$ small enough such that $\frac{1}{2\e} - \|\nabla_x u\|_{L^\infty} \geq \frac{C}{\e}$ for some $C>0$ independent of $\e$, we conclude the desired result.
\end{proof}

\begin{remark}\label{rmk_relax} The modulated kinetic energy estimate appeared in the proof of Proposition \ref{prop_ktoc2} can be also obtained by assuming $u \in L^\infty(0,T; \dot{W}^{1,\infty}(\om))\cap \dot{W}^{1,\infty}(0,T;L^\infty(\om))$, not $u\in L^\infty(\R^d \times (0,T))$. Indeed, we estimate $\sfJ_4$ as
\begin{align*}
\sfJ_4 &= \intrr (v-u) \cdot (\pa_t u) f^\e\,dxdv + \intrr (v-u) \cdot ((u-v) \cdot \nabla_x u) f^\e\,dxdv \cr
&\quad + \intrr (v-u) \cdot (v \cdot \nabla_x u) f^\e\,dxdv\cr
& \leq  \|\pa_t u\|_{L^\infty} \lt(\intrr |v-u|^2 f^\e\,dxdv\rt)^{1/2} + \|\nabla_x u\|_{L^\infty} \intrr  |v-u|^2 f^\e\,dxdv\cr
&\quad +\|\nabla_x u\|_{L^\infty} \lt(\intrr  |v-u|^2 f^\e\,dxdv\rt)^{1/2}\lt(\intrr  |v|^2 f^\e\,dxdv\rt)^{1/2}\cr
&\leq \lt(\frac{1}{2\e} + \|\nabla_x u\|_{L^\infty}\rt) \intrr |v-u|^2 f^\e\,dxdv +  \|\nabla_x u\|_{L^\infty}^2 \e\intrr  |v|^2 f^\e\,dxdv+  \e\|\pa_t u\|_{L^\infty}^2,
\end{align*}
due to $\intrr f^\e\,dxdv = 1$ for $t\geq0$. Then the kinetic energy in the right hand can be controlled by Lemma \ref{lem_free} uniformly in $\e$.
\end{remark}

\subsubsection{Proof of Theorem \ref{thm_ktoc}}

For the proof, we recall \cite[Lemma 2.1]{CC21} which gives relations between the bounded-Lipschitz distance and the modulated kinetic energy. 
\begin{lemma}\label{lem_conv}
\begin{itemize}
\item[(i)] Convergence of local moment:
\[
\d_{BL}\lt(\rho^\e u^\e, \, \rho u\rt) \leq \lt(\intrr |v - u(x) |^2 f^\e\,dxdv \rt)^{1/2} + C\d_{BL}(\rho^\e, \rho).
\]
\item[(ii)] Convergence of $f^\e$ towards the mono-kinetic distribution:
\[
d^2_{BL}(f^\e, \rho\otimes\delta_{u}) \leq C\intrr |v - u(x) |^2 f^\e\,dxdv + Cd^2_{BL}(\rho^\e, \rho).
\]
\end{itemize}
Here $C>0$ is independent of $\e$.
\end{lemma}
The proof of Theorem \ref{thm_ktoc} easily follows from a simple combination of \eqref{lem_conv} and Proposition \ref{prop_ktoc}.
%
%
%
%
%
%

%
%
%
%
%
%

\subsection{Hydrodynamic limit: from kinetic to Euler}\label{ssec_hydro} 
In this part, we now study the asymptotic analysis for the following kinetic equation when $\e \to 0$:
\[
\pa_t f^\e + v\cdot\nabla_x f^\e -  \nabla_v \cdot \lt((\gamma v +  \nabla K \star \rho^\e )f^\e \rt)  = \frac1\e\nabla_v \cdot ( (v - u^\e)f^\e +  c_P\nabla_v f^\e), 
\]
where $c_P =0$ or $1$, 
\[
\rho^\e = \intr f^\e\,dv, \quad \mbox{and} \quad \rho^\e u^\e = \intr vf^\e\,dv. 
\]
Similarly as before, we set $\gamma = 1$ without loss of generality in the rest of this part. 
As mentioned in Introduction, for the above asymptotic analysis, we employ the relative entropy method. For this, we first present the free energy estimates whose proof can be found in \cite[Lemma 2.1, Proposition 2.1]{CCJ21}. Let us recall the free energy $\mf$ and the dissipation $\md$:
\[
\mf(f) = c_P\intrr f \log f\,dxdv + \frac12\intrr |v|^2 f\,dxdv + \frac{1}{2}\intrr K(x-y)\rho(x) \rho(y)\,dxdy 
\]
and
\[
\md(f) = \intrr \frac{1}{f} \lt|c_P \nabla_v f - f(u-v) \rt|^2 dxdv.
\]

\begin{proposition}\label{prop_energy}
Suppose that $f$ is a solution of \eqref{main_eq} with sufficient integrability. Then we have
\[
\frac{d}{dt}\mf(f^\e) + \frac1\e \md_1(f^\e)   +  \intrr |v|^2 f^\e\,dxdv =  c_P \gamma d.
\]
Furthermore, we obtain
\[
\mf(f^\e) + \int_0^t \lt(\frac1{2\e}\md_1(f^\e) +  \intr \rho^\e|u^\e|^2\,dx \rt)ds \leq \mf(f_0^\e) + C(1+ \gamma^2)\e,
\]
where $C > 0$ depends only $T$ and $f_0^\e$.
\end{proposition}

Now, our main purpose is to prove the following quantitative bound estimate on the relative energies. 

\begin{proposition}\label{prop_hydro1} Let the assumptions of Theorem \ref{thm_hydro} verified.  Then we have the following inequalities for $0< \e \leq 1$ and $t \leq T$:
$$\begin{aligned}
&\frac12\intr \rho^\e |u^\e - u|^2\,dx + c_P\intr\mathcal{H}(\rho^\e|\rho)\,dx + \frac 12 \intr (\rho - \rho^\e)K\star(\rho - \rho^\e)\,dx  +  \int_0^t \intr \rho^\e| u^\e - u|^2\,dxds\cr
&\qquad \leq C\sqrt{\e}.
\end{aligned}$$
Here $C>0$ is a positive constant independent of $\e$. 
\end{proposition}

Since the strategy of the proof highly depends on the recent work \cite{CCJ21} combined with our main theorem, Theorem \ref{thm_main}, we briefly provide the details of that. As mentioned before, we use the relative entropy argument based on the weak-strong uniqueness principle to have the quantitative error estimates between the kinetic equation and the limiting system. 

We first begin with rewriting our desired limiting system \eqref{main_E} as a conservative form:
\[
\pa_t U + \nabla \cdot A(U) = F(U),
\]
where 
\[
 U := \begin{pmatrix}
\rho \\
m 
\end{pmatrix} 
\quad \mbox{with} \quad m = \rho u, \quad
A(U) := \begin{pmatrix}
m  & 0 \\
\displaystyle \frac{m \otimes m}{\rho} & c_P\rho \mathbb{I}_{d \times d}
\end{pmatrix},
\]
and
\[
F(U) := \begin{pmatrix}
0 \\
\displaystyle    -  \rho u -   \rho \nabla K \star \rho  
\end{pmatrix}.
\]
Here $\mathbb{I}_{d \times d}$ denotes the $d \times d$ identity matrix. The free energy of the above system is given by
\[
E(U) := \frac{|m|^2}{2\rho} + c_P \rho \log \rho.
\]
We now define the relative entropy functional $\me$ between two states of the system $U$ and $\bar U$ as follows.
\[
\me(\bar U|U) := E(\bar U) - E(U) - DE(U)(\bar U-U) \quad \mbox{with} \quad \bar U := \begin{pmatrix}
        \bar\rho \\
        \bar m \\
    \end{pmatrix}, \quad \bar m = \bar\rho \bar u,
\]
where $D E(U)$ denotes the derivation of $E$ with respect to $\rho, m$, i.e.,
$$\begin{aligned}
-DE(U)(\bar U - U) &= -\begin{pmatrix}
\displaystyle        -\frac{|m|^2}{2\rho^2} & c_P(\log \rho + 1)\\[3mm]
\displaystyle        \frac{m}{\rho} & 0
    \end{pmatrix}
    \begin{pmatrix}
    \bar\rho - \rho \\
    \bar m - m
    \end{pmatrix}\\
    &= \frac{\bar\rho |u|^2}{2} - \frac{\rho|u|^2}{2} + c_P(\rho - \bar\rho)(\log \rho + 1) + \rho |u|^2 - \bar\rho u \cdot \bar u.
\end{aligned}$$
This yields 
$$\begin{aligned}
\me(\bar U|U) &= \frac{\bar\rho|\bar u|^2}{2} - \frac{\rho|u|^2}{2} + c_P\bar\rho \log \bar\rho - c_P\rho \log \rho + \frac{\bar\rho |u|^2}{2} - \frac{\rho|u|^2}{2} + c_P(\rho - \bar\rho)(\log \rho + 1) + \rho |u|^2 - \bar\rho u \cdot \bar u\cr
&= \frac{\bar\rho}{2}|\bar u - u|^2 + c_P\mh(\bar\rho| \rho),
\end{aligned}$$
where $\mh(\bar\rho | \rho)$ is the relative entropy between densities given by 
\[
\mathcal{H}(\bar\rho|\rho):=  \int_{\rho}^{\bar\rho} \frac{\bar\rho - z}{z}\,dz =  \bar\rho  \log \bar\rho - \rho  \log \rho - (1+\log \rho) (\bar\rho-\rho).
\]

Then, by \cite[Proposition 3.1]{CCJ21}, we have the quantitative bounds on the relative entropy functional $\me$. 
\begin{proposition}\label{prop_re}Let the assumptions of Theorem \ref{thm_hydro} verified.  Then we have the following inequalities for $0< \e \leq 1$ and $t \leq T$:
$$\begin{aligned}
&\intr \me(U^\e|U)\,dx +  \int_0^t \intr \rho^\e(x)| u^\e(x) - u(x)|^2\,dxds\cr
&\qquad \leq C\sqrt{\e} + C\int_0^t \intr \me(U^\e|U)\,dxds  +\int_0^t \intr \rho^\e(x) ( u^\e(x) - u(x)) \cdot (\nabla K \star (\rho - \rho^\e))(x)\,dxds,
\end{aligned}$$
where $C>0$ is independent of $\e$.
\end{proposition}

\subsubsection{Proof of Proposition \ref{prop_hydro1}} We simply combine the bound estimate on the relative entropy $\me$ in Proposition \ref{prop_re} and the modulated interaction energy estimate in Theorem \ref{thm_main} to obtain
$$\begin{aligned}
&\intr \me(U^\e|U)\,dx + \frac 12 \intr (\rho - \rho^\e)K\star(\rho - \rho^\e)\,dx  + \int_0^t \intr \rho^\e| u^\e - u|^2\,dxds\cr
&\qquad \leq C\sqrt{\e}  + C \int_0^t \intr \me(U^\e|U)\,dxds + C\int_0^t \intr (\rho - \rho^\e)K\star(\rho - \rho^\e)\,dxds,
\end{aligned}$$
where $C>0$ is independent of $\e>0$. We finally apply the Gro\"onwall's lemma to the above to conclude the desired result.

\subsubsection{Proof of Theorem \ref{thm_hydro}} The proof of Theorem \ref{thm_hydro} readily follows from a combination of Proposition \ref{prop_hydro1} and the following lemma whose proofs can found in \cite[Lemma 2.2, Corollary 2.1 \& Corollary 2.2]{CCJ21}. 
\begin{lemma}\label{lem_conv} There exists a positive constant $C$ depending only on $\|u\|_{W^{1,\infty}}$ such that the following inequalities hold.
\begin{itemize}
\item[(i)] Error bound between densities:
\[
\|\rho^\e - \rho\|_{L^1} \leq \lt(2(\|\rho^\e\|_{L^1} + \|\rho\|_{L^1}) \rt)^{1/2}\lt(\intr \mh(\rho^\e| \rho)\,dx\rt)^{1/2}.
\]
\item[(ii)] Error bound between moments:
\[
\|\rho^\e u^\e - \rho u\|_{L^1} \leq \|\rho^\e\|_{L^1}^{1/2}\lt(\intr \rho^\e |u^\e - u|^2\,dx \rt)^{1/2} + \|u\|_{L^\infty}\|\rho^\e - \rho\|_{L^1}
\]
and
\[
\d_{BL}(\rho^\e u^\e, \rho u) \leq \|\rho^\e\|_{L^1}^{1/2}\lt(\intr \rho^\e |u^\e - u|^2\,dx \rt)^{1/2} + C\d_{BL}(\rho^\e, \rho).
\]
\item[(iii)] Error bound between convections:
$$\begin{aligned}
\|\rho^\e u^\e \otimes u^\e - \rho u \otimes u\|_{L^1} &\leq \intr \rho^\e |u^\e - u|^2\,dx + 2\|u\|_{L^\infty}\|\rho^\e\|_{L^1}^{1/2}\lt(\intr \rho^\e |u^\e - u|^2\,dx \rt)^{1/2}\cr
&\quad  + 3\|u\|_{L^\infty}^2\|\rho^\e - \rho\|_{L^1}
\end{aligned}$$
and
\[
\d_{BL}(\rho^\e u^\e \otimes u^\e, \rho u \otimes u) \leq \intr \rho^\e |u^\e - u|^2\,dx + C\|\rho^\e\|_{L^1}^{1/2}\lt(\intr \rho^\e |u^\e - u|^2\,dx \rt)^{1/2} + C\d_{BL}(\rho^\e, \rho).
\]

\item[(iv)] Error bound between particle distribution and Maxwellian ansatz:
\[
\lt\|f^\e - \frac{\rho}{(2\pi)^{d/2}} e^{-\frac{|u-v|^2}{2}}\rt\|_{L^1} \leq C  \lt(\intrr  \mathcal{H}\lt(f_0^\e \Big|\frac{\rho}{(2\pi)^{d/2}} e^{-\frac{|u-v|^2}{2}}\rt) dxdv\rt)^{1/2} + C\e^{1/8}.
\]
\item[(v)] Error bound between particle distribution and mono-kinetic ansatz:
$$\begin{aligned}
\d_{BL}(f^\e, \rho \otimes \delta_{u}) &\leq \|\rho^\e\|_{L^1}^{1/2}\lt( \lt(\intrr |v - u^\e|^2 f^\e\,dxdv\rt)^{1/2} + \lt(\intr  \rho^\e |u^\e - u|^2 \,dx\rt)^{1/2}\rt) \cr
&\quad +C  \d_{BL}(\rho^\e, \rho).
\end{aligned}$$
\end{itemize}
\end{lemma}

%
%
%
%
%
%

\section*{Acknowledgments}
We thank Jos\'e A. Carrillo for helpful conversations about the singular integrals with a polynomial operator. Y.-P. Choi was supported by National Research Foundation of Korea(NRF) grant funded by the Korea government(MSIP) (No. 2017R1C1B2012918) and Yonsei University Research Fund of 2020-22-0505. The work of J. Jung was supported by NRF grant (No. 2019R1A6A1A10073437).

%
%
%
%
%
%

\appendix

\section{Proof of Lemma \ref{lem_leib}}\label{app:lem_leib}

Since the cases $\ell=0,1$ are proved above, we only consider the induction step. Assuming the induction hypothesis, we use \eqref{rel_gam} to get
\begin{align*}
&\Delta_\gamma^{m+1}((u,0)\cdot\nabla_{(x,\xi)}\K) \\
&= \Delta_\gamma\lt( \sum_{\ell_1+\ell_2+p = m}\frac{m!}{\ell_1!\ell_2!p!} 2^p\sum_{j_1,\dots,j_p=1}^d (\Delta_x^{\ell_1} \pa_{j_1} \cdots \pa_{j_p}u,0) \cdot \nabla_{(x,\xi)} (\pa_{j_1}\cdots\pa_{j_p} \Delta_\gamma^{\ell_2}\K)\rt)\\
&= \sum_{\ell_1+\ell_2+p = m}\frac{m!}{\ell_1!\ell_2!p!} 2^p\sum_{j_1,\dots,j_p=1}^d (\Delta_x^{\ell_1+1} \pa_{j_1} \cdots \pa_{j_p}u,0) \cdot \nabla_{(x,\xi)} (\pa_{j_1}\cdots\pa_{j_p} \Delta_\gamma^{\ell_2}\K)\\
&\quad+ \sum_{\ell_1+\ell_2+p = m}\frac{m!}{\ell_1!\ell_2!p!} 2^{p+1}\sum_{j_1,\dots,j_{p+1}=1}^d (\Delta_x^{\ell_1} \pa_{j_1} \cdots \pa_{j_{p+1}}u,0) \cdot \nabla_{(x,\xi)} (\pa_{j_1}\cdots\pa_{j_{p+1}} \Delta_\gamma^{\ell_2}\K)\\
&\quad + \sum_{\ell_1+\ell_2+p = m}\frac{m!}{\ell_1!\ell_2!p!} 2^p\sum_{j_1,\dots,j_p=1}^d (\Delta_x^{\ell_1} \pa_{j_1} \cdots \pa_{j_p}u,0) \cdot \nabla_{(x,\xi)} (\pa_{j_1}\cdots\pa_{j_p} \Delta_\gamma^{\ell_2+1}\K)\\
&= \sum_{\substack{\ell_1+\ell_2+p = m \\ \ell_2, p\ge 1}}\frac{m!}{\ell_1!\ell_2!p!} 2^p\sum_{j_1,\dots,j_p=1}^d (\Delta_x^{\ell_1+1} \pa_{j_1} \cdots \pa_{j_p}u,0) \cdot \nabla_{(x,\xi)} (\pa_{j_1}\cdots\pa_{j_p} \Delta_\gamma^{\ell_2}\K)\\
&\quad+ \sum_{\substack{\ell_1+\ell_2+p = m\\ \ell_1, \ell_2 \ge 1}}\frac{m!}{\ell_1!\ell_2!p!} 2^{p+1}\sum_{j_1,\dots,j_{p+1}=1}^d (\Delta_x^{\ell_1} \pa_{j_1} \cdots \pa_{j_{p+1}}u,0) \cdot \nabla_{(x,\xi)} (\pa_{j_1}\cdots\pa_{j_{p+1}} \Delta_\gamma^{\ell_2}\K)\\
&\quad + \sum_{\substack{\ell_1+\ell_2+p = m\\ \ell_1, p \ge 1}}\frac{m!}{\ell_1!\ell_2!p!} 2^p\sum_{j_1,\dots,j_p=1}^d (\Delta_x^{\ell_1} \pa_{j_1} \cdots \pa_{j_p}u,0) \cdot \nabla_{(x,\xi)} (\pa_{j_1}\cdots\pa_{j_p} \Delta_\gamma^{\ell_2+1}\K)\\
&\quad + \sum_{\substack{\ell_1+\ell_2+p = m \\ \ell_2=0 \vee p=0}}\frac{m!}{\ell_1!\ell_2!p!} 2^p\sum_{j_1,\dots,j_p=1}^d (\Delta_x^{\ell_1+1} \pa_{j_1} \cdots \pa_{j_p}u,0) \cdot \nabla_{(x,\xi)} (\pa_{j_1}\cdots\pa_{j_p} \Delta_\gamma^{\ell_2}\K)\\
&\quad+ \sum_{\substack{\ell_1+\ell_2+p = m\\ \ell_1=0 \vee \ell_2 =0}}\frac{m!}{\ell_1!\ell_2!p!} 2^{p+1}\sum_{j_1,\dots,j_{p+1}=1}^d (\Delta_x^{\ell_1} \pa_{j_1} \cdots \pa_{j_{p+1}}u,0) \cdot \nabla_{(x,\xi)} (\pa_{j_1}\cdots\pa_{j_{p+1}} \Delta_\gamma^{\ell_2}\K)\\
&\quad + \sum_{\substack{\ell_1+\ell_2+p = m\\ \ell_1=0 \vee  p =0}}\frac{m!}{\ell_1!\ell_2!p!} 2^p\sum_{j_1,\dots,j_p=1}^d (\Delta_x^{\ell_1} \pa_{j_1} \cdots \pa_{j_p}u,0) \cdot \nabla_{(x,\xi)} (\pa_{j_1}\cdots\pa_{j_p} \Delta_\gamma^{\ell_2+1}\K)\\
&= \sum_{\substack{\ell_1+\ell_2+p = m+1\\ \ell_1, \ell_2, p\ge 1}}\frac{(m+1)!}{\ell_1!\ell_2!p!} 2^p\sum_{j_1,\dots,j_p=1}^d (\Delta_x^{\ell_1} \pa_{j_1} \cdots \pa_{j_p}u,0) \cdot \nabla_{(x,\xi)} (\pa_{j_1}\cdots\pa_{j_p} \Delta_\gamma^{\ell_2}\K)\\
&\quad + \sum_{\substack{\ell_1+\ell_2+p = m+1 \\ (\ell_1\ge 1)\wedge( \ell_2=0 \vee p=0)}}\frac{m!}{(\ell_1-1)!\ell_2!p!} 2^p\sum_{j_1,\dots,j_p=1}^d (\Delta_x^{\ell_1} \pa_{j_1} \cdots \pa_{j_p}u,0) \cdot \nabla_{(x,\xi)} (\pa_{j_1}\cdots\pa_{j_p} \Delta_\gamma^{\ell_2}\K)\\
&\quad+ \sum_{\substack{\ell_1+\ell_2+p = m+1\\ (p\ge 1)\wedge(\ell_1=0 \vee \ell_2 =0)}}\frac{m!}{\ell_1!\ell_2!(p-1)!} 2^{p}\sum_{j_1,\dots,j_p=1}^d (\Delta_x^{\ell_1} \pa_{j_1} \cdots \pa_{j_p}u,0) \cdot \nabla_{(x,\xi)} (\pa_{j_1}\cdots\pa_{j_p} \Delta_\gamma^{\ell_2}\K)\\
&\quad + \sum_{\substack{\ell_1+\ell_2+p = m+1\\ (\ell_2\ge 1)\wedge( \ell_1=0 \vee  p =0)}}\frac{m!}{\ell_1!(\ell_2-1)!p!} 2^p\sum_{j_1,\dots,j_p=1}^d (\Delta_x^{\ell_1} \pa_{j_1} \cdots \pa_{j_p}u,0) \cdot \nabla_{(x,\xi)} (\pa_{j_1}\cdots\pa_{j_p} \Delta_\gamma^{\ell_2}\K)\\
&=: \sum_{i=1}^4 \sfI_i.
\end{align*}
We next estimate the terms $\sfI_i,i=2,3,4$ as
\begin{align*}
\sfI_2&= \sum_{\substack{\ell_1+p = m+1 \\ \ell_1\ge 1}}\frac{m!}{(\ell_1-1)!p!} 2^p\sum_{j_1,\dots,j_p=1}^d (\Delta_x^{\ell_1} \pa_{j_1} \cdots \pa_{j_p}u,0) \cdot \nabla_{(x,\xi)} (\pa_{j_1}\cdots\pa_{j_p} \K)\\
&\quad + \sum_{\substack{\ell_1+\ell_2 = m+1 \\ \ell_1\ge 1}}\frac{m!}{(\ell_1-1)!\ell_2!}  (\Delta_x^{\ell_1} u,0) \cdot \nabla_{(x,\xi)} \Delta_\gamma^{\ell_2}\K) - (\Delta_x^{(m+1)} u,0)\cdot\nabla_{(x,\xi)} \K\\
&=  \sum_{\substack{\ell_1+p = m+1 \\ \ell_1\ge 1}}\frac{m!}{(\ell_1-1)!p!} 2^p\sum_{j_1,\dots,j_p=1}^d (\Delta_x^{\ell_1} \pa_{j_1} \cdots \pa_{j_p}u,0) \cdot \nabla_{(x,\xi)} (\pa_{j_1}\cdots\pa_{j_p} \K)\\
&\quad + \sum_{\substack{\ell_1+\ell_2 = m+1 \\ \ell_1,\ell_2\ge 1}}\frac{m!}{(\ell_1-1)!\ell_2!}  (\Delta_x^{\ell_1} u,0) \cdot \nabla_{(x,\xi)} \Delta_\gamma^{\ell_2}\K),
\end{align*}
\begin{align*}
\sfI_3&= \sum_{\substack{\ell_2+p = m+1\\ p\ge 1}}\frac{m!}{\ell_2!(p-1)!} 2^{p}\sum_{j_1,\dots,j_p=1}^d ( \pa_{j_1} \cdots \pa_{j_p}u,0) \cdot \nabla_{(x,\xi)} (\pa_{j_1}\cdots\pa_{j_p} \Delta_\gamma^{\ell_2}\K)\\
&\quad + \sum_{\substack{\ell_1+p = m+1\\ p\ge 1}}\frac{m!}{\ell_1!(p-1)!} 2^{p}\sum_{j_1,\dots,j_p=1}^d (\Delta_x^{\ell_1} \pa_{j_1} \cdots \pa_{j_p}u,0) \cdot \nabla_{(x,\xi)} (\pa_{j_1}\cdots\pa_{j_p} \K)\\
&\quad - 2^{m+1} \sum_{j_1,\dots,j_{m+1}=1}^d ( \pa_{j_1} \cdots \pa_{j_{m+1}}u,0) \cdot \nabla_{(x,\xi)} (\pa_{j_1}\cdots\pa_{j_{m+1}} \K)\\
&= \sum_{\substack{\ell_2+p = m+1\\ p\ge 1}}\frac{m!}{\ell_2!(p-1)!} 2^{p}\sum_{j_1,\dots,j_p=1}^d ( \pa_{j_1} \cdots \pa_{j_p}u,0) \cdot \nabla_{(x,\xi)} (\pa_{j_1}\cdots\pa_{j_p} \Delta_\gamma^{\ell_2}\K)\\
&\quad + \sum_{\substack{\ell_1+p = m+1\\ p\ge 1, \ell_1\ge1}}\frac{m!}{\ell_1!(p-1)!} 2^{p}\sum_{j_1,\dots,j_p=1}^d (\Delta_x^{\ell_1} \pa_{j_1} \cdots \pa_{j_p}u,0) \cdot \nabla_{(x,\xi)} (\pa_{j_1}\cdots\pa_{j_p} \K),
\end{align*}
and
\begin{align*}
\sfI_4&= \sum_{\substack{\ell_2+p = m+1\\ \ell_2\ge 1}}\frac{m!}{(\ell_2-1)!p!} 2^p\sum_{j_1,\dots,j_p=1}^d ( \pa_{j_1} \cdots \pa_{j_p}u,0) \cdot \nabla_{(x,\xi)} (\pa_{j_1}\cdots\pa_{j_p} \Delta_\gamma^{\ell_2}\K)\\
&\quad +\sum_{\substack{\ell_1+\ell_2 = m+1\\ \ell_2\ge 1}}\frac{m!}{\ell_1!(\ell_2-1)!}  (\Delta_x^{\ell_1} u,0) \cdot \nabla_{(x,\xi)} ( \Delta_\gamma^{\ell_2}\K) - (u,0)\cdot\nabla_{(x,\xi)}\Delta_\gamma^{m+1}\K\\
&= \sum_{\substack{\ell_2+p = m+1\\ \ell_2\ge 1}}\frac{m!}{(\ell_2-1)!p!} 2^p\sum_{j_1,\dots,j_p=1}^d ( \pa_{j_1} \cdots \pa_{j_p}u,0) \cdot \nabla_{(x,\xi)} (\pa_{j_1}\cdots\pa_{j_p} \Delta_\gamma^{\ell_2}\K)\\
&\quad +\sum_{\substack{\ell_1+\ell_2 = m+1\\ \ell_2\ge 1, \ell_1\ge1}}\frac{m!}{\ell_1!(\ell_2-1)!}  (\Delta_x^{\ell_1} u,0) \cdot \nabla_{(x,\xi)} ( \Delta_\gamma^{\ell_2}\K).
\end{align*}
Combining those estimates yields
\begin{align*}
\sfI_2 + \sfI_3 + \sfI_4 &= \sum_{\substack{\ell_1+p = m+1 \\ \ell_1\ge 1, p\ge 1}}\frac{(m+1)!}{\ell_1!p!} 2^p\sum_{j_1,\dots,j_p=1}^d (\Delta_x^{\ell_1} \pa_{j_1} \cdots \pa_{j_p}u,0) \cdot \nabla_{(x,\xi)} (\pa_{j_1}\cdots\pa_{j_p} \K)\\
&\quad + (\Delta_x^{(m+1)} u,0)\cdot\nabla_{(x,\xi)} \K\\
&\quad + \sum_{\substack{\ell_1+\ell_2 = m+1 \\ \ell_1\ge 1, \ell_2 \ge 1}}\frac{(m+1)!}{\ell_1!\ell_2!}  (\Delta_x^{\ell_1} u,0) \cdot \nabla_{(x,\xi)} \Delta_\gamma^{\ell_2}\K) \\
&\quad +  (u,0)\cdot\nabla_{(x,\xi)} \Delta_\gamma^m \K\\
&\quad +\sum_{\substack{\ell_2+p = m+1\\ \ell_2\ge1, p\ge 1}}\frac{(m+1)!}{\ell_2!p!} 2^{p}\sum_{j_1,\dots,j_p=1}^d ( \pa_{j_1} \cdots \pa_{j_p}u,0) \cdot \nabla_{(x,\xi)} (\pa_{j_1}\cdots\pa_{j_p} \Delta_\gamma^{\ell_2}\K)\\
&\quad + 2^{m+1} \sum_{j_1,\cdots,j_{m+1}=1}^d ( \pa_{j_1} \cdots \pa_{j_{m+1}}u,0) \cdot \nabla_{(x,\xi)} (\pa_{j_1}\cdots\pa_{j_{m+1}} \K)\\
&= \sum_{\substack{\ell_1+\ell_2+p = m+1\\ \ell_1\ell_2 p =0}}\frac{(m+1)!}{\ell_1!\ell_2!p!} 2^p\sum_{j_1,\dots,j_p=1}^d (\Delta_x^{\ell_1} \pa_{j_1} \cdots \pa_{j_p}u,0) \cdot \nabla_{(x,\xi)} (\pa_{j_1}\cdots\pa_{j_p} \Delta_\gamma^{\ell_2}\K).
\end{align*}
Therefore, we can conclude that $\sum_{i=1}^4 \sfI_i$ gives the desired result.

%
%
%
%
%
%

\end{document}